\documentclass[a4paper, 11pt, reqno]{amsart}

\usepackage{amsmath,amssymb,amsthm}
\usepackage{thmtools}
\usepackage{enumerate}
\usepackage{xcolor}
\usepackage[mathscr]{euscript}
\usepackage{url}
\usepackage[all]{xy}
\usepackage{booktabs}
\usepackage{graphicx}
\usepackage{epstopdf}
\usepackage{algorithmic}
\usepackage[caption=false]{subfig}
\usepackage{mathtools}
\usepackage{bigdelim}

\usepackage[center]{caption}
\usepackage{makecell}

\usepackage[misc]{ifsym}
\usepackage{multirow}
\usepackage{float}
\usepackage{placeins}

\usepackage[ruled,vlined]{algorithm2e}

\usepackage{appendix}

\usepackage{hyperref}

\definecolor{cobalt}{rgb}{0.0, 0.28, 0.67}
\definecolor{darkblue}{rgb}{0.0, 0.0, 0.55}

\hypersetup
{
bookmarksopen=true,
pdftitle={Globally convergent Block Coordinate Descent Algorithms for Semi-Nonnegative Independent Component Analysis},
pdfauthor={},
pdfsubject={Simplex-constrained optimization for semi-nonnegative independent component analysis},
pdfmenubar=true,
pdfhighlight=/O,
colorlinks=true,
pdfpagelayout=SinglePage,
pdffitwindow=true,
linkcolor=cobalt,
citecolor=cobalt,
urlcolor=cobalt
}

\usepackage[capitalize,nameinlink,noabbrev]{cleveref}

\crefname{appendix}{Appendix}{Appendices}
\Crefname{appendix}{Appendix}{Appendices}

\newenvironment{subalgorithm}[1][H]{
\begin{algorithm}[#1]
\SetAlgorithmName{Subalgorithm}{Subalgorithm}{List of Subalgorithms}
\crefalias{algocf}{subalgorithm}
}{
\end{algorithm}
}

\crefname{subalgorithm}{Subalgorithm}{Subalgorithms}
\Crefname{subalgorithm}{Subalgorithm}{Subalgorithms}

\usepackage[scale=0.75]{geometry}
\usepackage{tikz}
\usetikzlibrary{shapes}
\usetikzlibrary{arrows}
\usetikzlibrary{matrix}
\theoremstyle{plain}
\newtheorem{theorem}{Theorem}[section]
\newtheorem{corollary}[theorem]{Corollary}
\newtheorem{lemma}[theorem]{Lemma}
\newtheorem{proposition}[theorem]{Proposition}
\theoremstyle{definition}

\newtheorem{assumption}{Assumption}

\crefname{assumption}{Assumption}{Assumptions}
\Crefname{assumption}{Assumption}{Assumptions}

\theoremstyle{remark}
\newtheorem{remark}[theorem]{Remark}
\newcommand{\tens}[1]{\boldsymbol{\mathcal{#1}}}

\newcommand{\matr}[1]{\boldsymbol{#1}}

\newcommand{\vect}[1]{\boldsymbol{#1}}

\newcommand{\set}[1]{\mathcal{#1}}
\newcommand{\kr}{\odot}     % Khatri-Rao product
\newcommand{\hadam}{\circ}    % Hadamard product
\newcommand{\kron}{\otimes}   % Kronecker product
\newcommand{\T}{\mathsf{T}}        % transposition
\newcommand{\herm}{\mathsf{H}}     % Hermitian transposition
\newcommand{\vecl}{\operatorname{vec}}
\makeatletter
\newcommand{\rank}[1]{\mathop{\operator@font rank}\!\left\{#1\right\}}
\newcommand{\colrank}[1]{\mathop{\operator@font colrank}\!\left\{#1\right\}}
\newcommand{\krank}[1]{\mathop{\operator@font krank}\!\left\{#1\right\}}
\newcommand{\trace}[1]{\mathop{\operator@font tr}\!\left(#1\right)}
\newcommand{\Diag}[1]{\mathop{\operator@font Diag}\!\left\{#1\right\}}    % a diagonal matrix
\newcommand{\diag}[1]{\mathop{\operator@font \textbf{diag}}\!\left\{#1\right\}}  
\newcommand{\blockdiag}[1]{\mathop{\operator@font \textbf{blockdiag}}\!\left\{#1\right\}} % a vector
\newcommand{\Span}[1]{\mathop{\operator@font Span}\!\left\{#1\right\}}    % a space
\newcommand{\argmin}{\mathop{\operator@font argmin}}

\newcommand{\offdiag}[1]{\mathop{\operator@font offdiag}\!\left\{#1\right\}}    % a vector
\newcommand{\Proj}[1]{\operatorname{Proj}_{#1}}
\newcommand{\ProjGrad}[2]{\mathop{{\operator@font grad}}#1\!\left(#2\right)}
\newcommand{\symmm}[1]{\mathop{\operator@font sym}\!\left(#1\right)}
\newcommand{\skeww}[1]{\mathop{\operator@font skew}\!\left(#1\right)}
\makeatother

\newcommand{\eqdef}{\stackrel{\sf def}{=}}

\usepackage{tikz}
\usetikzlibrary{shapes.geometric, arrows}

\tikzstyle{startstop} = [rectangle, rounded corners, minimum width=1.8cm, minimum height=0.65cm,text centered, draw=black, fill=green!30]
\tikzstyle{startstop0} = [rectangle, rounded corners, minimum width=0.01cm, minimum height=0.65cm,text centered, draw=white, fill=white!30]
\tikzstyle{io} = [trapezium, trapezium left angle=70, trapezium right angle=110, minimum width=3cm, minimum height=1cm, text centered, draw=black, fill=blue!30]
\tikzstyle{process} = [rectangle, minimum width=3cm, minimum height=1cm, text centered, draw=black, fill=orange!30]
\tikzstyle{decision} = [diamond, minimum width=3cm, minimum height=1cm, text centered, draw=black, fill=green!30]
\tikzstyle{arrow} = [thick,->,>=stealth]

\makeatletter
\def\bbordermatrix#1{\begingroup \m@th
\@tempdima 4.75\p@
\setbox\z@\vbox{%
\def\cr{\crcr\noalign{\kern2\p@\global\let\cr\endline}}%
\ialign{$##$\hfil\kern2\p@\kern\@tempdima&\thinspace\hfil$##$\hfil
&&\quad\hfil$##$\hfil\crcr
\omit\strut\hfil\crcr\noalign{\kern-\baselineskip}%
#1\crcr\omit\strut\cr}}%
\setbox\tw@\vbox{\unvcopy\z@\global\setbox\@ne\lastbox}%
\setbox\tw@\hbox{\unhbox\@ne\unskip\global\setbox\@ne\lastbox}%
\setbox\tw@\hbox{$\kern\wd\@ne\kern-\@tempdima\left[\kern-\wd\@ne
\global\setbox\@ne\vbox{\box\@ne\kern2\p@}%
\vcenter{\kern-\ht\@ne\unvbox\z@\kern-\baselineskip}\,\right]$}%
\null\;\vbox{\kern\ht\@ne\box\tw@}\endgroup}
\makeatother

\makeatletter
\@ifundefined{subjclassname@2020}
{\@namedef{subjclassname@2020}{\textup{2020} Mathematics Subject Classification}}
{}
\makeatother

\begin{document}

\title[Globally convergent Block Coordinate Descent Algorithms]
{Globally convergent Block Coordinate Descent Algorithms for Semi-Nonnegative Independent Component Analysis}

\author{Zhenyu Cai}
\address{School of Science, Sun Yat-sen University,
Shenzhen, Guangdong, China}
\email{caizhy27@mail2.sysu.edu.cn}

\author{Wentao Ding}
\address{Department of Systems Engineering and Engineering Management,
The Chinese University of Hong Kong, Shatin, N.T., Hong Kong}
\email{wentaoding@cuhk.edu.hk}

\author{Jianze Li}
\address{School of Science, Sun Yat-sen University,
Shenzhen, Guangdong, China}
\email{lijianze@mail.sysu.edu.cn}

\author{Wen Huang}
\address{School of Mathematical Sciences, Xiamen University, Xiamen, China}
\email{wen.huang@xmu.edu.cn}

\keywords{semi-nonnegative independent component analysis,
probability simplex,
block coordinate descent,
elementary column transformation,
global convergence,
\L{}ojasiewicz inequality}

\subjclass[2020]{90C26, 65K05, 15A23, 94A12}

\begin{abstract}
We propose a subgradient-based block coordinate descent (BCD-SubG)
framework for nonconvex optimization over products of probability
simplices and a closed feasible set, motivated by semi-nonnegative
independent component analysis (Semi-NNICA).
Elementary column transformations are used to preserve simplex feasibility,
while first-order information guides block and pair selection.
We establish the global convergence of BCD-SubG using a
constrained \L{}ojasiewicz inequality; additional eventual-interiority conditions yield stationary limits and convergence
rates. 
We then specialize BCD-SubG to the direct-fit and indirect-fit
Semi-NNICA models and derive explicit updates that ensure global
convergence without further assumptions. For both models, eventual
interiority also yields stationary limits and convergence-rate estimates.
Experiments on synthetic mixtures and simulated magnetic resonance
spectroscopy data show generally lower recovery errors and shorter
wall-clock times than the tested Semi-NNICA baselines. Joint diagonalization
experiments further demonstrate lower final objective values.
\end{abstract}

\maketitle

\section{Introduction}

In this paper, we adopt the \emph{block coordinate descent} (BCD) framework to study the following optimization problem:
\begin{align}
\min \; f(\matr{U}, \matr{V}), 
\quad (\matr{U}, \matr{V}) \in \Delta_{n}^{m} \times \Omega,
\label{eq:cost_general}
\end{align}
where $m\geq1$, $n\geq2$, and $\Delta_{n}^{m}$ denotes the Cartesian product of $m$ \emph{probability
simplices} in $\mathbb{R}^n$ defined by
\[
\Delta_{n} \eqdef \left\{ \vect{u} \in \mathbb{R}^{n}, \; u_i \geq 0,\; \sum_{i=1}^{n} u_i = 1 \right\},
\]
and $\matr{U}=\left(\vect{u}^{(1)},\vect{u}^{(2)},\ldots,\vect{u}^{(m)}\right)\in\Delta_{n}^{m}$. 
Equivalently, $\Delta_n^m$ is the set of $n\times m$ \emph{column-stochastic matrices}, also called \emph{left-stochastic matrices}~\cite{hogben2006handbook}.
Here, $\Omega$ is a closed feasible set in a finite-dimensional Euclidean space, and the second block satisfies $\matr{V}\in\Omega$. The choice of $\Omega$ is left general, while the precise assumptions on its local analytic parametrization and the $\matr{V}$-block update are stated in \cref{thm:Schneider_frame_two}. 
In addition, the cost function $f$ is assumed to admit a
real-analytic extension to an open neighborhood of
$\Delta_n^m\times\Omega$.
We mainly focus on the iterative update of the matrix $\matr{U}$ over the set $\Delta_{n}^{m}$.

The primary motivation for studying problem \eqref{eq:cost_general}
comes from \emph{independent component analysis} (ICA), a standard
framework for \emph{blind source separation} (BSS)
\cite{comon1994independent,Como10:book}.
Nonnegativity may be imposed on the sources
\cite{plumbley2003algorithms} or, as in Semi-NNICA, on the mixing
matrix \cite{wang2014nonnegative}.
Classical ICA algorithms include cumulant-based joint diagonalization
and fixed-point methods \cite{cardoso1993blind,hyvarinen1999fast}.
The classical \emph{linear instantaneous mixture model} of ICA is formulated as:
\begin{align}\label{Ica_model}
\vect{x}(t) = \matr{A}\vect{s}(t)+\vect{\varepsilon}(t),\ \ t = 1,2,\ldots, T, 
\end{align}
where $\vect{x}(t) = [x_1(t), \ldots, x_n(t)]^\T \in \mathbb{R}^n$ denotes the \emph{observed signals}, 
$\vect{s}(t) = [s_1(t), \ldots, s_m(t)]^\T \in \mathbb{R}^m$ denotes the \emph{source signals}, 
$\matr{A} \in \mathbb{R}^{n \times m}$ is an unknown \emph{mixing matrix}, and $\vect{\varepsilon}(t)$ denotes the noise. 
Given the observed signals $\vect{x}(t)$, the goal of BSS or ICA is to recover the underlying source signals $\vect{s}(t)$ and the mixing matrix $\matr{A}$.
In this paper, we consider the overdetermined setting, where $n\geq m$; that is, the number of observed signals is greater than or equal to the number of source signals. 
% As commonly assumed for the identifiability of ICA, the mixing matrix
% $\matr{A}$ in the ICA model is assumed to have full column rank.
In many practical scenarios, additional structural constraints are required to better capture the underlying physical or statistical properties of the data. 
Among these, \emph{semi-nonnegative independent component analysis} (Semi-NNICA)
\cite{wang2014nonnegative} has received considerable attention due to
its practical relevance. In Semi-NNICA, the mixing matrix $\matr{A}$ is
constrained to be element-wise nonnegative, while the source signals are
not necessarily nonnegative. Let $\set{D}_m$ denote the space of real
$m\times m$ diagonal matrices and write
$\matr{D}=(\matr{D}^{(1)},\ldots,\matr{D}^{(L)})\in(\set{D}_m)^L$.
Two widely used Semi-NNICA criteria \cite{coloigner2012line,coloigner2014semi,coloigner2014line,wang2013nonnegative,wang2014canonical,wang2013new}, \cite[Eqs. (1.29), (2.46)]{wang2014nonnegative} are
\begin{align}
\tilde{f}_{\mathrm{dir}}(\matr{A},\matr{D})
&= \sum_{\ell=1}^{L} \left\|\, \matr{C}^{(\ell)} - \matr{A}\matr{D}^{(\ell)}\matr{A}^{\T} \,\right\|^{2}, 
\label{eq:cost_direct-old}\\
\tilde{f}_{\mathrm{indir}}(\matr{A})
&= \sum_{\ell=1}^{L} \left\| \operatorname{offdiag}\!\left(\matr{A}^{\T} \bigl(\matr{C}^{(\ell)}\bigr)^{-1} \matr{A}\right) \right\|^{2}, 
\label{eq:cost_indirect-old}
\end{align}
where $\matr{A}\geq\matr{0}$.
Here, $\matr{C}^{(\ell)}\in\mathbb{R}^{n\times n}$ denotes the $\ell$-th
real symmetric matrix slice of the cumulant tensor $\tens{C}$ associated with the
observed signals in the model \eqref{Ica_model}, and $L$ denotes the
total number of such slices. The direct-fit formulation \eqref{eq:cost_direct-old} is considered for the
general case $n\geq m$.
For the indirect-fit formulation \eqref{eq:cost_indirect-old}, throughout
this paper we restrict attention to the square case $n=m$ and assume
that each $\matr{C}^{(\ell)}$ is nonsingular.

In the direct-fit model \eqref{eq:cost_direct-old}, the nonnegativity
constraint can be imposed directly on the mixing matrix.
In contrast, the indirect-fit objective \eqref{eq:cost_indirect-old} is
homogeneous in $\matr{A}$ and therefore suffers from a scaling
degeneracy. An additional normalization is thus required to prevent the
objective from being reduced merely by scaling $\matr{A}$ toward zero.
%In this work, the column-wise probability-simple constraint is adopted to simultaneously enforce nonnegativity and fix the scale of each mixing vector.
More details about these two models can be found in \cite{wang2014nonnegative} and the references therein. 
To handle the nonnegativity constraints in the models \eqref{eq:cost_direct-old} and \eqref{eq:cost_indirect-old}, a common strategy in Semi-NNICA is to transform the constrained optimization problem into an unconstrained one via element-wise squaring \cite[Eq. (1.52)]{wang2014nonnegative} or exponential mappings \cite{coloigner2014semi}. 
Based on these reformulations, various algorithms, including Jacobi-type algorithms, \emph{line search} and \emph{trust region} strategies, have been systematically 
studied \cite{coloigner2012line,coloigner2014semi,coloigner2014line,wang2014nonnegative,wang2013nonnegative,wang2014canonical,wang2013new}. 
The square parametrization increases the polynomial degree
\cite[Eq.~(2.22)]{wang2014nonnegative}, while the exponential
parametrization introduces nonpolynomial nonlinearities.
In the experiments reported in
\cite[Tables~3.1 and~3.2]{wang2014nonnegative},
$\mathrm{JDC}^{+}_{\mathrm{ADMM-3}}$, based on the alternating
direction method of multipliers (ADMM), requires fewer floating-point
operations than $\mathrm{JDC}^{+}_{\mathrm{LU-3}}$.
These results motivate methods that handle nonnegativity directly,
but do not imply a general complexity ordering between the two approaches.

In this paper, we take a different approach by constraining the matrix $\matr{A}$ in the models \eqref{eq:cost_direct-old} and \eqref{eq:cost_indirect-old} 
to lie in a Cartesian product of probability simplices $\Delta_{n}^{m}$, as in \eqref{eq:cost_general}.
This constraint is natural, as it enforces both nonnegativity and normalization on the columns of $\matr{A}$, and leads to a more structured optimization problem. 
For the direct-fit criterion, we also introduce a quadratic penalty on the
diagonal factors with a fixed parameter $\lambda>0$. We thus consider
\begin{align}
\min_{\matr{A}\in\Delta_n^m,\,\matr{D}\in(\set{D}_m)^L}
\quad f_{\mathrm{dir}}(\matr{A},\matr{D})
&\eqdef \tilde{f}_{\mathrm{dir}}(\matr{A},\matr{D})
+\lambda\sum_{\ell=1}^{L}\|\matr{D}^{(\ell)}\|^2,
\label{eq:cost_direct}\\
\min_{\matr{A}\in\Delta_n^m}\quad f_{\mathrm{indir}}(\matr{A})
&= \sum_{\ell=1}^{L} \left\| \operatorname{offdiag}\!\left(\matr{A}^{\T} \bigl(\matr{C}^{(\ell)}\bigr)^{-1} \matr{A}\right) \right\|^{2}. 
\label{eq:cost_indirect}
\end{align}
With the column scales fixed by the simplex constraints, the penalty controls
the magnitude of the diagonal factors and discourages large, mutually
cancelling components when mixing columns are nearly dependent. 
Under this formulation, the direct-fit model \eqref{eq:cost_direct}
and, in the square case $n=m$, the indirect-fit model
\eqref{eq:cost_indirect} can be viewed as special cases of
\eqref{eq:cost_general}, corresponding to the choices
$\Omega=(\set{D}_m)^L$ and
$\Omega=\{\matr{0}_{1\times1}\}$, respectively.

To solve models~\eqref{eq:cost_direct} and~\eqref{eq:cost_indirect}, we adopt a
\emph{subgradient-based block coordinate descent} (BCD-SubG) framework, in
which each selected block is updated via an \emph{elementary column transformation} (ECT);
see \cref{sec:bcd-framework} for further details. 
An ECT transfers mass between two coordinates according to
$\vect{u}^{+}
=
\vect{u}
+
t\bigl(\vect{e}_i-\vect{e}_j\bigr)$ with 
$t\in[-u_i,u_j],$ 
and thus preserves simplex feasibility. 
Closely related mechanisms appear
in two-coordinate descent methods for linearly constrained
problems~\cite{beck2014twocoordinate,cristofari2019almost,konnov2019selective},
in sequential minimal optimization (SMO), where the dual equality constraint
is preserved during support vector machine (SVM) training~\cite{platt1998sequential,fan2005working},
and in pairwise variants of Frank--Wolfe (FW), where weight is transferred between two
vertices~\cite{lacostejulien2015global}.
Related gradient-based Jacobi-type methods for orthogonal and unitary
diagonalization use first-order information to select transformations
and establish convergence through \L{}ojasiewicz inequalities
\cite{li2018globally,ULC2019}.
Our contribution is to integrate this
mass-transfer mechanism into a BCD framework involving both simplex and
closed-set blocks, to use first-order gradient information to select the updated block and ordered coordinate pair, and to establish the \emph{global convergence}\footnote{In this paper, global convergence means that, from any initial point, the entire sequence of iterates converges to a single limit point.} of the
resulting algorithm.

A wide range of optimization algorithms over $\Delta_n^m$ have been developed in the literature and can be broadly classified into three main categories. 
First, methods in the original variables include columnwise projected
gradient (PG) with simplex projections \cite{condat2016fast} and
entropy-based mirror descent \cite{beck2003mirror}.
A particular multiplicative-weights variant gives second-order
stationarity guarantees for almost all initializations under the
stated assumptions and sufficiently small step sizes
\cite{panageas2019multiplicative}.
Second, block methods exploit the product structure. Block-coordinate
FW (Block-FW) updates one simplex factor through a vertex oracle;
its randomized variant achieves an expected $\mathcal{O}(1/k)$
objective-gap rate for smooth convex problems
\cite{lacostejulien2013block}.
Recent projection-free methods also cover nonconvex product-domain
problems, with local active-set identification for suitable away-step
variants under strict complementarity \cite{bomze2025projection}.
Proximal alternating linearized minimization (PALM) handles mixed
domains $\Delta_n^m\times\set{C}$ through block-separable indicators;
under its standard blockwise Lipschitz and Kurdyka--\L{}ojasiewicz
assumptions, bounded iterates converge to a critical point
\cite{bolte2014proximal}.
Third, smooth reformulations use the Hadamard lift
$\matr{X}=\matr{Z}\hadam\matr{Z}$ from
$(\mathbb{S}^{n-1})^m$ onto $\Delta_n^m$
\cite{li2023simplex,levin2025effect}. Although lifted first-order points
may be spurious at zero coordinates, second-order points map to
stationary points of the simplex-constrained problem.
Riemannian proximal-gradient methods also address nonsmooth composites
when the associated proximal subproblems are tractable
\cite{huang2022riemannian}.

We now use a simple example to illustrate why ECT is adopted in this paper to solve the Semi-NNICA models~\eqref{eq:cost_direct} and~\eqref{eq:cost_indirect}.
Consider the following nonconvex simplex-constrained problem:
\begin{equation}
\min_{\vect{u}\in\Delta_3}
\quad
f(\vect{u})
=
-2u_1^3+u_1^2-u_1-u_2 .
\label{eq:simplex-motivation}
\end{equation}
Starting from $\vect{u}^{(0)}=(0.15,0.82,0.03)^\T$, the classical PG trajectory shown in \cref{fig:simplex-motivation} converges to the suboptimal local minimizer $\vect{e}_2$, where $f(\vect{e}_2)=-1$. By contrast, the ECT strategy employed in the proposed BCD-SubG method reaches the global minimizer $\vect{e}_1$ after only two ECT updates, attaining $f(\vect{e}_1)=-2$. Thus, even on a three-dimensional simplex, ECT can more readily move the iterates out of an unfavorable local region and thus converge to a better solution.
%The plotted paths are qualitative: in particular, a PG trajectory depends on the chosen step rule.

\begin{figure}[t]
\centering
\includegraphics[width=0.4\textwidth]
{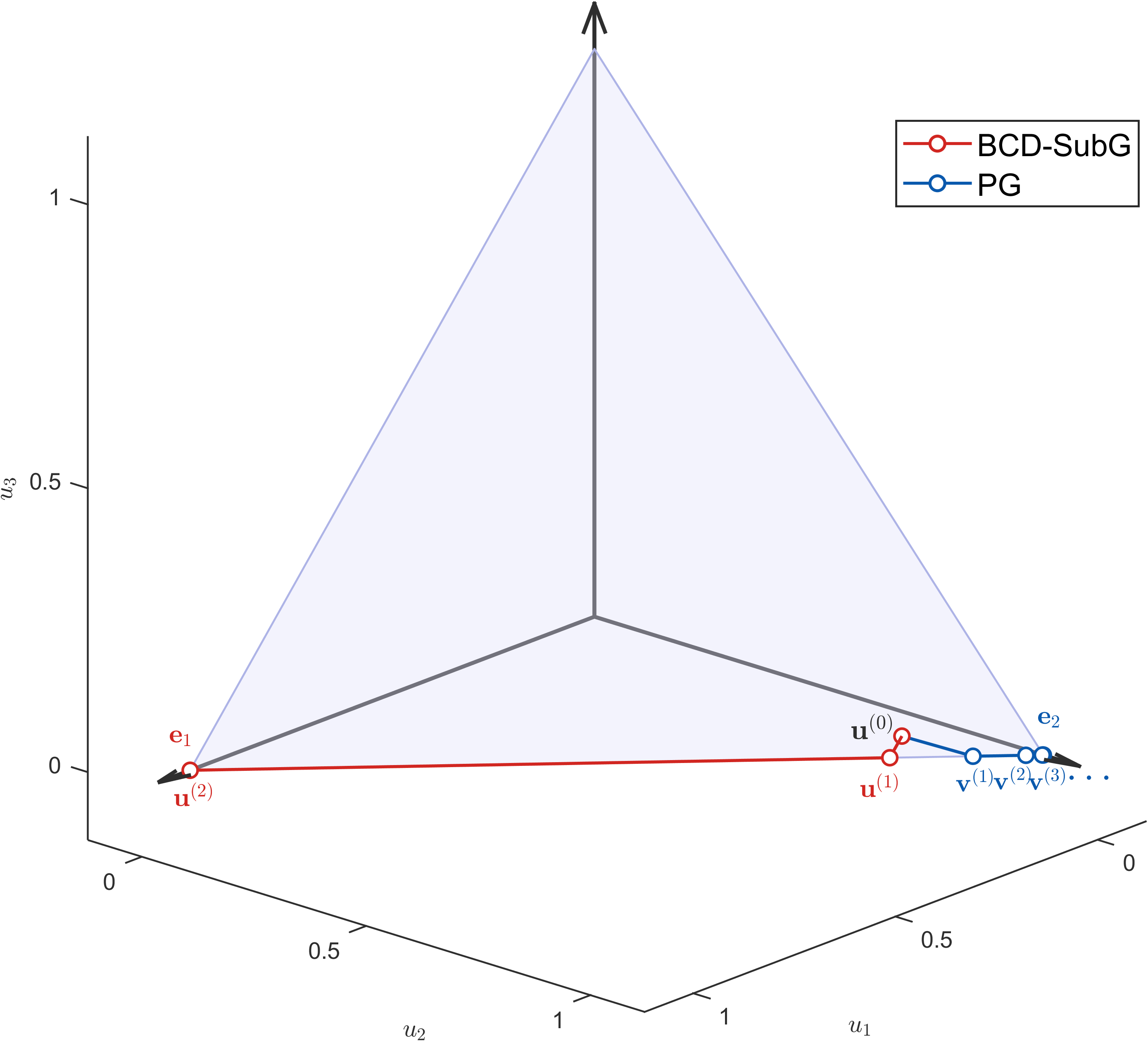}
\caption{Iteration trajectories of PG and BCD-SubG for the illustrative simplex-constrained problem~\eqref{eq:simplex-motivation}.}
\label{fig:simplex-motivation}
\end{figure}

In this paper, our main contributions are as follows:
\begin{itemize}
\item We propose BCD-SubG for solving~\eqref{eq:cost_general}.
The framework uses first-order gradient information to select the updated
block and an ordered coordinate pair. An ECT update transfers mass
between the selected coordinates while preserving simplex feasibility.
\item Under the stated analytic, descent, and zero-step conditions,
we prove that any iterate sequence with an accumulation point
converges as a whole. Additional step-length and regularity conditions,
together with eventual interiority, yield stationarity and convergence
rates governed by the constrained \L{}ojasiewicz exponent.
\item We specialize BCD-SubG to the direct-fit and indirect-fit Semi-NNICA
models and derive explicit update rules.
For both models, we establish global convergence of BCD-SubG
without further assumptions. Under eventual interiority, we also
obtain stationarity of the limit and convergence-rate estimates.
\item Experiments on synthetic mixtures and simulated magnetic
resonance spectroscopy data show generally lower recovery errors
and shorter wall-clock time than the tested Semi-NNICA baselines.
In joint diagonalization, at each tested SNR, one BCD-SubG variant
attains the lowest median final objective value among the compared methods.
\end{itemize}

The rest of the paper is organized as follows. \Cref{sec:prelimi} introduces
the stationarity measure and the \L{}ojasiewicz framework. The general
algorithm and its convergence analysis follow in \cref{sec:bcd-framework}.
\Cref{sec:semi-nnica-direct,sec:semi-nnica} specialize the method to the two
Semi-NNICA models, and \cref{sec:numerical_experiments} reports the numerical
results. \Cref{sec:discussion} concludes the paper.

\section{Preliminaries}\label{sec:prelimi}

\subsection{Notations}

Bold uppercase letters denote matrices and matrix-valued block tuples,
bold lowercase letters denote vectors, and scalars are not bold.
Script letters denote sets; we retain the conventional notation
$\Delta_n$, $\Omega$, and $\mathbb{R}^n$ for the simplex, the second
feasible set, and Euclidean space, respectively.
Matrix entries, such as $A_{i,p}$, are scalars. 
The tuple $\matr{D}=(\matr{D}^{(1)},\ldots,\matr{D}^{(L)})$ is a block
variable, not a tensor variable. The relations $\geq$ and $\succeq$ denote
entrywise and positive-semidefinite inequalities for matrices, respectively.
Throughout this paper, $\|\cdot\|$ denotes the Euclidean norm for vectors
and the Frobenius norm for matrices, unless otherwise specified.
For a block tuple $\matr{X}=(\matr{X}^{(1)},\ldots,\matr{X}^{(r)})$, we set
$\|\matr{X}\|^2\eqdef\sum_{i=1}^{r}\|\matr{X}^{(i)}\|^2$.
The inner product $\langle\cdot,\cdot\rangle$ is the corresponding
Euclidean, Frobenius, or product inner product. For a point $\matr{x}$ and a nonempty set $\mathcal{S}$, we define
$\operatorname{dist}(\matr{x},\mathcal{S})
\eqdef
\inf_{\matr{y}\in\mathcal{S}}\|\matr{x}-\matr{y}\|$.
For a matrix $\matr{X}$, $\vecl(\matr{X})$ denotes its column-wise
vectorization; for a block tuple, $\vecl$ stacks the vectorized blocks.
We identify matrices and block tuples with these Euclidean coordinates
when applying the abstract convergence results.
The vector $\operatorname{diag}(\matr{X})$ contains the diagonal entries
of $\matr{X}$, and $\operatorname{Diag}(\vect{x})$ is the diagonal matrix
with diagonal $\vect{x}$. We set
$\operatorname{offdiag}(\matr{X})
\eqdef\matr{X}-\operatorname{Diag}(\operatorname{diag}(\matr{X}))$.
Finally, $\kron$, $\kr$, and $\hadam$ denote the Kronecker, Khatri--Rao,
and Hadamard products, respectively.

\subsection{Tangent cones and constrained stationarity}

Let $f: \set{D} \to \mathbb{R}$ be continuously differentiable on an open set $\set{D}\subseteq\mathbb{R}^{N}$, and let $\set{M} \subseteq \set{D}$ be a closed feasible set. 
The Bouligand tangent cone of $\set{M}$ at $\vect{x}\in\set{M}$ is
\[
\set{T}_{\vect{x}}\set{M}
\eqdef
\left\{\vect{v}:\ \exists\,t_\nu\downarrow0,\ \vect{x}_\nu\in\set{M},\
\vect{x}_\nu\to\vect{x},\
(\vect{x}_\nu-\vect{x})/t_\nu\to\vect{v}\right\}.
\]
For a general closed set, the tangent cone need not be convex, and
stationarity is measured by $g_f^-$ as defined below.
When $\set{M}$ is also convex, the polar cone of
$\set{T}_{\vect{x}}\set{M}$ is its \emph{normal cone}
$\set{N}_{\vect{x}}\set{M}$ \cite{rockafellar1998variational}.
In this case, the \emph{constrained subdifferential} of $f$ on $\set{M}$ is defined as
$\partial_{\set{M}} f(\vect{x}) \eqdef \nabla f(\vect{x}) + \set{N}_{\vect{x}}\set{M}$.
A point $\vect{x} \in \set{M}$ is a \emph{stationary point} of $f$ on $\set{M}$ if $-\nabla f(\vect{x}) \in \set{N}_{\vect{x}}\set{M}$; equivalently, we have $\vect{0} \in \partial_{\set{M}} f(\vect{x})$. The projection of the origin onto
$\partial_{\set{M}} f(\vect{x})$, denoted by
$\Proj{\partial_{\set{M}} f(\vect{x})}(\vect{0})$, is the unique
minimum-norm element of $\partial_{\set{M}} f(\vect{x})$ and is referred
to as the \emph{minimum-norm subgradient}. Hence, $\|
\Proj{\partial_{\set{M}} f(\vect{x})}(\vect{0})
\|
=
\operatorname{dist}\!\left(
\vect{0},\partial_{\set{M}} f(\vect{x})
\right)$. Moreover, Moreau's decomposition gives $\Proj{\partial_{\set{M}} f(\vect{x})}(\vect{0})
=
-\Proj{\set{T}_{\vect{x}}\set{M}}
\bigl(-\nabla f(\vect{x})\bigr)$. Consequently,
$-\Proj{\partial_{\set{M}} f(\vect{x})}(\vect{0})
\in \set{T}_{\vect{x}}\set{M}$,
showing that the negative minimum-norm subgradient is a first-order
feasible direction at $\vect{x}\in\set{M}$.

We now consider the $\set{M}=\Delta_n$ case. For any $\vect{u} \in \Delta_n$, the constrained subdifferential of $f$ on $\Delta_n$ can be defined as
\begin{align}\label{eq:subdifferential_simplex}
\partial_{\Delta_n} f(\vect{u}) = \nabla f(\vect{u}) + \set{N}_{\vect{u}}\Delta_n  =\left\{\vect{g}+ \mu \vect{1} - \sum_{i=1}^n \gamma_i \vect{e}_i \;\middle|\; \mu \in \mathbb{R},\; \gamma_i \geq 0,\; \gamma_i u_i = 0 , \; 1\leq i\leq n \right\},
\end{align}
where $\vect{g}=\nabla f(\vect{u})=(g_1,\ldots,g_n)^{\T}\in\mathbb{R}^n$ is the Euclidean gradient of $f$ at $\vect{u}$, $\vect{1} \in \mathbb{R}^n$ denotes the vector of all ones, and $\vect{e}_i \in \mathbb{R}^n$ is the $i$-th standard basis vector. 
The minimum-norm subgradient $\Proj{\partial_{\Delta_n} f(\vect{u})}(\vect{0})$, denoted by $\vect{\xi}_{\min}=(\xi_1,\ldots,\xi_n)^{\T}\in \mathbb{R}^n$, is the unique minimum-norm element of $\partial_{\Delta_n} f(\vect{u})$ characterized in \eqref{eq:subdifferential_simplex}. The complementarity condition $\gamma_i u_i = 0$ forces $\gamma_i^* = 0$ when $u_i>0$. For $u_i = 0$, minimizing the squared magnitude $(g_i + \mu - \gamma_i)^2$ subject to $\gamma_i \geq 0$ requires projecting $(g_i + \mu^*)$ onto the non-negative orthant, yielding $\gamma_i^* = \max(0,\, g_i + \mu^*)$ for a given $\mu^*$.
Consequently, $\vect{\xi}_{\min}$ admits the following componentwise form: 
\begin{equation}\label{eq:eta_form}
\xi_i =
\begin{cases}
g_i + \mu^*, & \text{if } u_i > 0, \\
\min(0,\, g_i + \mu^*), & \text{if } u_i = 0,
\end{cases}
\end{equation}
where the dual variable $\mu^* \in \mathbb{R}$ is determined by $\sum_{i=1}^n \xi_i = 0$ to ensure $-\vect{\xi}_{\min} \in \set{T}_{\vect{u}}\Delta_n$.

\subsection{\L{}ojasiewicz inequality}
Let $f: \set{D} \to \mathbb{R}$ be continuously differentiable on an open set $\set{D}\subseteq\mathbb{R}^{N}$, and let $\set{M} \subseteq \set{D}$ be a closed feasible set. 
As in \cite{schneider2015convergence}, we define $g^{-}_f(\vect{x})$, the norm of any metric projection of $-\nabla f(\vect{x})$ onto $\set{T}_{\vect{x}}\set{M}$, by
\begin{equation*}
g^{-}_f(\vect{x}) \eqdef \sqrt{ \left\|\nabla f(\vect{x})\right\|^2 - \operatorname{dist}\!\left(-\nabla f(\vect{x}), \set{T}_{\vect{x}}\set{M}\right)^2 }.
\end{equation*}
For a general closed feasible set $\set{M}$, we call
$\vect{x}\in\set{M}$ \emph{stationary} if $g^-_f(\vect{x})=0$.
In the unconstrained setting, this measure is the Euclidean gradient
norm $\|\nabla f(\vect{x})\|$. On a closed convex set, it equals
$\operatorname{dist}\!\left(\vect{0},\partial_{\set{M}}f(\vect{x})\right)$.
In particular, on a linear subspace it is the norm of the Euclidean
gradient of the restricted function.

The next lemma, adapted from \cite[Prop.~2.2]{schneider2015convergence}, extends the analytic \L{}ojasiewicz inequality to closed
feasible sets that admit a suitable local analytic parametrization.
\begin{lemma}[Constrained analytic \L{}ojasiewicz inequality]
\label{lem:nonsmooth-lojasiewicz}
Let $\set{D}\subseteq\mathbb{R}^{N}$ be open, let
$\set{M}\subseteq\set{D}$ be closed, and let
$f:\set{D}\to\mathbb{R}$ be real analytic. Fix
$\vect{x}^*\in\set{M}$. Assume that there exist an integer $q\geq1$,
an open set $\set{U}\subseteq\mathbb{R}^{q}$, a point
$\vect{z}^*\in\set{U}$, and a componentwise real-analytic map
$\tau:\set{U}\to\set{D}$ such that
\begin{itemize}
\item[(i)] $\tau(\set{U})\subseteq\set{M}$ and
$\tau(\vect{z}^*)=\vect{x}^*$;
\item[(ii)] for every open neighborhood
$\set{U}'\subseteq\set{U}$ of $\vect{z}^*$,
$\tau(\set{U}')$ contains a relatively open neighborhood of
$\vect{x}^*$ in $\set{M}$.
\end{itemize}
Then there exist $r_{\mathrm{L}},C_{\mathrm{L}}>0$ and $\theta\in(0,1/2]$ such that
every $\vect{x}\in\set{M}$ with
$\|\vect{x}-\vect{x}^*\|<r_{\mathrm{L}}$ satisfies
\begin{equation}\label{eq:lojasiewicz}
\left|f(\vect{x}) - f(\vect{x}^*)\right|^{1-\theta} \leq C_{\mathrm{L}}\,g^{-}_f(\vect{x}).\tag{\textbf{\L{}}}
\end{equation}   
\end{lemma}

Using the above \cref{lem:nonsmooth-lojasiewicz} and the 
\L{}ojasiewicz inequality in 
\eqref{eq:lojasiewicz}, the following convergence theorem was established in \cite[Thm. 2.3]{schneider2015convergence}.

\begin{theorem}[Abstract global convergence]\label{thm:Schneider_frame}
Let $\set{M} \subseteq \set{D}$ be a closed set and $f: \set{D}\to \mathbb{R}$ be a real-analytic function. 
Suppose that for large enough $k$, the sequence $\{\vect{x}_k\}\subseteq \set{M}$ satisfies the following conditions:
\begin{itemize}
\item[(i)] There exists $\eta > 0$ such that: 
\begin{equation*}\label{cond:A1}
f(\vect{x}_{k-1}) - f(\vect{x}_{k}) \geq \eta\,g^{-}_f(\vect{x}_{k-1}) \left\|\vect{x}_{k} - \vect{x}_{k-1}\right\|. \tag{\textbf{A1}}
\end{equation*}
\item[(ii)]\leavevmode\vspace{-\baselineskip}
\begin{equation*}\label{cond:A2}
g^{-}_f(\vect{x}_{k-1})=0
\quad\Longrightarrow\quad
\vect{x}_{k}=\vect{x}_{k-1}.\tag{\textbf{A2}}
\end{equation*}
\end{itemize}
If the sequence $\{\vect{x}_k\}$ has an accumulation point 
$\vect{x}^*\in\set{M}$ satisfying \eqref{eq:lojasiewicz}, then the whole sequence converges to $\vect{x}^*$. 
Furthermore, suppose that the following additional
condition is satisfied:
\begin{itemize}
\item[(iii)] There exists $\kappa > 0$ such that: 
\begin{equation*}\label{cond:A3}
\left\|\vect{x}_{k} - \vect{x}_{k-1}\right\| \geq \kappa\,g^{-}_f(\vect{x}_{k-1}). \tag{\textbf{A3}}
\end{equation*}
\end{itemize}
Then $g^{-}_f(\vect{x}_k)\to 0$, and the convergence rate
can be estimated by
\begin{equation}\label{eq:conver_rate}
\left\|\vect{x}_k - \vect{x}^*\right\| \leq 
\begin{cases} 
C_{\mathrm{rate}}e^{-c_{\mathrm{rate}}k}, & \text{if } \theta = \frac{1}{2} \text{ (linear convergence)}, \\ 
C_{\mathrm{rate}}k^{-\frac{\theta}{1-2\theta}}, & \text{if } 0 < \theta < \frac{1}{2} \text{ (sublinear convergence)},
\end{cases}
\end{equation}
where $C_{\mathrm{rate}},c_{\mathrm{rate}}>0$ are independent of $k$, and $\theta$ denotes the \L{}ojasiewicz exponent of $f$ at $\vect{x}^*$ as defined in 
\eqref{eq:lojasiewicz}. 
\end{theorem}

\section{General BCD-SubG framework}\label{sec:bcd-framework}

%\subsection{Model Reformulation}

\subsection{Elementary column transformations}

Let $(i,j)$ satisfy $1\leq i,j\leq n$ and $i\neq j$. For
$\alpha\in\mathbb{R}$, define
\begin{equation}
\matr{T}_{i,j}(\alpha) \eqdef 
\matr{I}_n+\alpha(\vect{e}_i-\vect{e}_j)\vect{e}_j^{\T}.
\label{eq:elementary_column-stochastic}
\end{equation}
Every column of $\matr{T}_{i,j}(\alpha)$ sums to one, so we call it an
\emph{ECT}. For
$\alpha\in[0,1]$, the matrix is also column-stochastic. Let
$\vect{u}=(u_1,\ldots,u_n)^{\T}\in\Delta_n$ and
$\vect{u}'=\matr{T}_{i,j}(\alpha)\vect{u}$. Its entries satisfy
\begin{equation}\label{eq:ij_transformation}
u'_{k} = 
\begin{cases} 
u_{i} + \alpha u_{j}, &\text{if } k = i, \\ 
(1 - \alpha)u_{j}, &\text{if } k = j, \\ 
u_{k}, &\text{if } k \neq i, j. 
\end{cases}
\end{equation}
Suppose $a\eqdef u_{i} \geq 0$ and $b\eqdef u_{j} > 0$. It follows directly that $\vect{u}' \in \Delta_{n}$ if and only if $\alpha \in \left[-\frac{a}{b}, 1\right]$.  
In addition, when $\alpha<0$, we set $\widetilde{\alpha}=-\frac{u_j}{u_i}\alpha\in (0,1]$. Then $\matr{T}_{i,j}(\alpha)$ and $\matr{T}_{j,i}(\widetilde{\alpha})$ produce the same update. 
Hence, every feasible ECT transformation admits an equivalent representation 
with a nonnegative coefficient when interchanging the indices \(i\) and \(j\). 

Beyond preserving feasibility, ECT transformations also provide a \emph{finite connectivity} property on the simplex.
More precisely, let $\vect{u},\vect{v}\in\Delta_n$ with
$\vect{u}\neq\vect{v}$. Then there exist a suitable index $i$ with
$u_i\neq v_i$, an index $j\neq i$, and a feasible parameter $\alpha$
such that $\bigl(\matr{T}_{i,j}(\alpha)\vect{u}\bigr)_i=v_i$. 
By applying this argument successively to the unmatched coordinates, for any $\vect{u},\vect{v}\in\Delta_n$, there exist index
pairs $(i_1,j_1),\ldots,(i_q,j_q)$ and feasible parameters
$\alpha_1,\ldots,\alpha_q$, with $q\leq n-1$, such that
$\vect{v}
=
\matr{T}_{i_q,j_q}(\alpha_q)
\cdots
\matr{T}_{i_1,j_1}(\alpha_1)\vect{u}.$ 
This finite connectivity property provides a theoretical justification for us to employ ECT transformations to perform updates on $\Delta_n$.

When requiring $\alpha\in[0,1]$, the ECT transformation in \eqref{eq:elementary_column-stochastic} can be interpreted as a pairwise mass-redistribution step in a Markov process~\cite{seneta2006non}, where such transformations are typically used to describe the evolution of a probability distribution over states and to characterize its stationary or long-term stochastic behavior.
However, the purpose of introducing ECT transformations in this paper is different. Rather than modeling stochastic dynamics, we employ them as structured update operators to optimize the objective function in \eqref{eq:cost_general}.

We now introduce the definition of the elementary functions associated with the ECT transformations. 
Let $f : \Delta_{n}^{m} \times \Omega \to \mathbb{R}$ be defined as in the general model~\eqref{eq:cost_general}, and 
$(\matr{U}, \matr{V}) \in \Delta_{n}^{m} \times \Omega$.
We denote the two \emph{restricted functions} by
$f_{\matr{V}}:\matr{U}\mapsto f(\matr{U},\matr{V})$ and
$f_{\matr{U}}:\matr{V}\mapsto f(\matr{U},\matr{V})$.
For each $1\leq p\leq m$, we define the restriction of $f_{\matr{V}}$ to the $p$-th block variable as
\begin{equation}\label{eq:p_block}
f_{(p;\matr{U}^{(-p)},\matr{V})} : \Delta_{n}\to\mathbb{R},\ \ 
\vect{u}\mapsto f_{\matr{V}}\left(
\vect{u}^{(1)},\ldots,\vect{u}^{(p-1)},
\vect{u},
\vect{u}^{(p+1)},\ldots,\vect{u}^{(m)}
\right),
\end{equation}
where
$\matr{U}^{(-p)}
\eqdef
\left(
\vect{u}^{(1)},\ldots,\vect{u}^{(p-1)},
\vect{u}^{(p+1)},\ldots,\vect{u}^{(m)}
\right)
\in \Delta_n^{m-1}$.
When optimizing $f_{(p;\matr{U}^{(-p)},\matr{V})}$ via the ECT transformation process in \eqref{eq:elementary_column-stochastic}, the update along a given pair $(i,j)$ can be characterized by a one-dimensional parametric function with respect to $\alpha$. Specifically, for a current iterate $\vect{u}^{(p)} \in \Delta_{n}$, we define the associated \emph{elementary function} as
\begin{equation}
h^{(i,j)}_{(\matr{U}^{(-p)}, \vect{u}^{(p)}, \matr{V})}(\alpha) \eqdef 
f_{(p;\matr{U}^{(-p)},\matr{V})}\left(\matr{T}_{i,j}(\alpha)\vect{u}^{(p)}\right).
\label{eq:elementary_function}
\end{equation}
For simplicity, we write \eqref{eq:p_block} as $f_{(p)}(\vect{u}^{(p)})$ and \eqref{eq:elementary_function} as $h^{(i,j)}_{(p)}(\alpha)$ when the restricted variables are clear from the context. Throughout this paper, $\partial h^{(i,j)}_{(p)}(\alpha)$ denotes the
constrained subdifferential of $h^{(i,j)}_{(p)}$ over its feasible
interval; that is, we have
\begin{equation}\label{eq:elementary_subdifferential}
    \partial h^{(i,j)}_{(p)}(\alpha)
\eqdef
\partial_{[-u_i^{(p)}/u_j^{(p)},\,1]}
h^{(i,j)}_{(p)}(\alpha).
\end{equation}

For an elementary function $h(\alpha)$ associated with the index pair $(i,j)$,
$\left[-\frac{a}{b},1\right]$ is the feasible
interval induced by the ECT transformation in \eqref{eq:ij_transformation}. We consider the following two step-size
selection rules:
\begin{itemize}
\item[\hypertarget{rule:R1}{(R1)}] \emph{Exact line search:}
If $\operatorname{dist}\!\left(0,\partial h(0)\right)=0$, we set $\alpha^*=0$.
Otherwise, $\alpha^*$ is chosen as a global minimizer of $h$ over
$\left[-\frac{a}{b},1\right]$.

\item[\hypertarget{rule:R2}{(R2)}] \emph{First local line search:}
If $\operatorname{dist}\!\left(0,\partial h(0)\right)=0$, we set $\alpha^*=0$.
Otherwise, $\alpha^*$ is chosen as the first local minimizer of $h$
encountered from $\alpha=0$ along the descent direction determined by
$h'(0)$ over $\left[-\frac{a}{b},1\right]$.
\end{itemize}

\begin{remark}
When employing ECT transformations to minimize the objective function in \eqref{eq:cost_general}, a natural idea is to apply them directly to the entire product space $\Delta_n^m$, rather than updating one column at a time. 
A major difficulty with this approach is that, for $m\geq2$, applying the same ECT to all columns does not connect
arbitrary pairs of points in $\Delta_n^m$. 
Indeed, every common left
multiplication preserves equality between columns. Hence, a matrix
with two identical columns cannot be transformed into one in which
the corresponding columns differ. This motivates updating the
simplex factors separately.
\end{remark}

\subsection{BCD-SubG algorithm}

We solve \eqref{eq:cost_general} with the BCD-SubG framework.
At each iteration, a block is first selected according to first-order stationarity measures. 
If the $\matr{U}$-block is selected, an index pair $(i,j)$ is determined by enforcing an additional subgradient inequality condition. 
Then an update is performed via the ECT transformation $\matr{T}_{i,j}(\alpha)$. 
The resulting algorithm is described in \cref{alg:gradient-based}.

\begin{algorithm}[H]
\caption{BCD-SubG algorithm.}
\label{alg:gradient-based}
\KwIn{A function $f : \Delta_{n}^{m} \times \Omega \to \mathbb{R}$ as in the general model \eqref{eq:cost_general}, a parameter $0 < \delta \leq 1/[n\sqrt{m(n-1)}]$, and a starting point $(\matr{U}_{0}, \matr{V}_{0})\in\Delta_n^m\times\Omega$, where $\matr{U}_{0}=\bigl(\vect{u}_{0}^{(1)},\ldots,\vect{u}_{0}^{(m)}\bigr)$.}
\KwOut{Sequence of iterations $\{(\matr{U}_{k}, \matr{V}_{k})\}$.}
\begin{itemize}
\item \textbf{For} $k = 1, 2, \dots$ [until a stopping criterion is satisfied] do
\item If 
\begin{align}\label{eq:update_selection_inequality}
g^{-}_{f_{\matr{U}_{k-1}}}(\matr{V}_{k-1})
\;\geq\;
\operatorname{dist}\!\left(\matr{0},\,\partial_{\Delta_{n}^{m}} f_{\matr{V}_{k-1}}(\matr{U}_{k-1})\right),
\end{align}
we set $\matr{U}_{k}=\matr{U}_{k-1}$, and perform an update over $\Omega$ to obtain $\matr{V}_{k}$. 
\item Otherwise, we set $\matr{V}_{k}=\matr{V}_{k-1}$, and perform an update over $\Delta_{n}^{m}$.
\begin{itemize}
\item[$\bullet$] Choose $p_k\in\{1,\ldots,m\}$ and an index pair
$(i_k,j_k)$ with $i_k\neq j_k$ and
$(\vect{u}_{k-1}^{(p_k)})_{j_k}>0$ such that
\begin{align}\label{eq:grad_inequality_h}
\operatorname{dist}\!\left(0,\,\partial h_{k}(0)\right)\geq \delta 
\operatorname{dist}\!\left(\matr{0},\,\partial_{\Delta_{n}^{m}} f_{\matr{V}_{k-1}}(\matr{U}_{k-1})\right),
\end{align}
where $h_k(\alpha)\eqdef h^{(i_k,j_k)}_{(\matr{U}_{k-1}^{(-p_k)}, \vect{u}_{k-1}^{(p_k)}, \matr{V}_{k-1})}(\alpha)$. 
\item[$\bullet$] Compute $\alpha_k^*$ by applying
\hyperlink{rule:R1}{Rule~(R1)} or
\hyperlink{rule:R2}{Rule~(R2)} to $h_k$.
\item[$\bullet$] Set $\matr{U}^{(-p_k)}_k=\matr{U}^{(-p_k)}_{k-1}$ and update $\vect{u}^{(p_k)}_k=\matr{T}_{i_k,j_k}(\alpha_k^*)\vect{u}^{(p_k)}_{k-1}$.
\end{itemize}
\item \textbf{End for}
\end{itemize}
\end{algorithm}

\begin{proposition}[Well-defined block and pair selection]
\label{prop:block-pair-existence}
For every $0<\delta\leq1/[n\sqrt{m(n-1)}]$, there exist $p_k$ and
$(i_k,j_k)$ satisfying \eqref{eq:grad_inequality_h}.
\end{proposition}

The above proposition shows that the simplex-selection step is well defined.
Its proof, deferred to \cref{app:block-pair-selection}, first selects a block
with a large minimum-norm subgradient and then a pair with a large feasible
directional derivative.

\subsection{Convergence Analysis}\label{subsec:conver_analy}

Based on \cref{thm:Schneider_frame} in \cref{sec:prelimi}, we establish the convergence of the iterate sequence $\{(\matr{U}_k,\matr{V}_k)\}$ generated by \cref{alg:gradient-based}. The constrained \L{}ojasiewicz inequality on $\Delta_n^m\times\Omega$ used in the analysis is established in \cref{lem:product-lojasiewicz} in \cref{app:general-convergence}.

Before establishing the convergence result, we introduce \cref{assump:interiority}, which will be applied to the relevant block variables when needed. This assumption is only required to handle the boundary issue arising in the verification of condition \eqref{cond:A3}. It does not affect the verification of \eqref{cond:A1} and \eqref{cond:A2}, and hence is not required for the global convergence result.

\begin{assumption}\label{assump:interiority}
There exists a constant $\varepsilon_{\mathrm{int}}>0$ such that, for all sufficiently
large $k$, the simplex iterate $\matr{U}_k\in\Delta_n^m$ satisfies
$(\matr{U}_k)_{i,p}\geq\varepsilon_{\mathrm{int}}$ for all $i,p$.
\end{assumption}

The next theorem combines the descent and step-length controls for the two
blocks. It reduces the analysis to conditions~\eqref{cond:A1}--\eqref{cond:A3}
and then applies the constrained \L{}ojasiewicz inequality.
\begin{theorem}[Convergence of BCD-SubG]\label{thm:Schneider_frame_two}
Let $f$ be as in \eqref{eq:cost_general}, where $\Omega$ is a closed feasible set such that, at every
$\matr{V}^*\in\Omega$, it admits a real-analytic parametrization
satisfying conditions~(i)--(ii) of
\cref{lem:nonsmooth-lojasiewicz}, and $f$ admits a
real-analytic extension to a neighborhood of $\Delta_n^m\times\Omega$.
Let $\{(\matr{U}_k,\matr{V}_k)\}$ be the sequence generated by 
\cref{alg:gradient-based}. 
Suppose that there exist constants $\eta_{\mathrm{U}},\eta_{\mathrm{V}}>0$ such that, for 
sufficiently large $k$, the following conditions hold.
\begin{itemize}
\item[(i)] When the inequality \eqref{eq:update_selection_inequality} holds, we have 
\begin{align*}\label{cond:V1}
f_{\matr{U}_{k-1}}(\matr{V}_{k-1})
-
f_{\matr{U}_{k-1}}(\matr{V}_{k})
&\geq
\eta_{\mathrm{V}}
g^{-}_{f_{\matr{U}_{k-1}}}(\matr{V}_{k-1})
\left\|
\matr{V}_{k}-\matr{V}_{k-1}
\right\|.\tag{\textbf{A1-V}}\\
\label{cond:V2}
g^{-}_{f_{\matr{U}_{k-1}}}(\matr{V}_{k-1})=0
\ &\Longrightarrow\ 
\matr{V}_{k}=\matr{V}_{k-1}.\tag{\textbf{A2-V}}
\end{align*}
\item[(ii)] Otherwise, we have 
\begin{align*}\label{cond:U1}
h_k(0)
-
h_k(\alpha_k^*)
&\geq
\eta_{\mathrm{U}}
\operatorname{dist}\!\left(
0,\partial h_k(0)
\right)
|\alpha_k^*|.\tag{\textbf{A1-U}}\\
\label{cond:U2}
\operatorname{dist}\!\left(
0,\partial h_k(0)
\right)=0
\ &\Longrightarrow\ 
\alpha_k^*=0.\tag{\textbf{A2-U}}
\end{align*}
\end{itemize}
If the sequence $\{(\matr{U}_k,\matr{V}_k)\}$ admits an accumulation
point $(\matr{U}^*,\matr{V}^*)\in\Delta_n^m\times\Omega$, then the
whole sequence converges to $(\matr{U}^*,\matr{V}^*)$.
Furthermore, suppose that \cref{assump:interiority} holds for the
$\matr{U}$-block, that $(\matr{U},\matr{V})\mapsto g^-_{f_{\matr{U}}}(\matr{V})$ is lower
semicontinuous, and that there exist constants $\kappa_{\mathrm{U}},\kappa_{\mathrm{V}}>0$ such that the following additional conditions are satisfied for large enough $k$: 

\begin{itemize}
\item[(iii)] when the inequality \eqref{eq:update_selection_inequality} holds, we have  
\begin{equation*}\label{cond:V3}
\left\|\matr{V}_{k}-\matr{V}_{k-1}\right\|
\geq
\kappa_{\mathrm{V}}
g^{-}_{f_{\matr{U}_{k-1}}}(\matr{V}_{k-1}).\tag{\textbf{A3-V}}
\end{equation*}
\item[(iv)] otherwise, we have
\begin{equation*}\label{cond:U3}
|\alpha_k^*|
\geq
\kappa_{\mathrm{U}}
\operatorname{dist}\!\left(
0,\partial h_k(0)
\right).\tag{\textbf{A3-U}}
\end{equation*}
\end{itemize}
Then $(\matr{U}^*,\matr{V}^*)$ is stationary, and convergence holds at the rate given in \eqref{eq:conver_rate} of 
\cref{thm:Schneider_frame}.
\end{theorem}

The proof is given in \cref{app:general-convergence}.

\subsection{Convergence Analysis for Semi-NNICA}\label{subsec:conver_analy_ICA}

We now apply \cref{alg:gradient-based} to the direct-fit and
indirect-fit Semi-NNICA models in
\eqref{eq:cost_direct} and \eqref{eq:cost_indirect}, respectively,
and specialize \cref{thm:Schneider_frame_two} to the resulting
subalgorithms. The update estimates used below are established
in \cref{sec:semi-nnica-direct,sec:semi-nnica}. The proofs of both corollaries
are given in \cref{app:proof-semi-nnica-convergence}.

For the direct-fit model, the fixed quadratic penalty guarantees bounded
iterates and uniform strong convexity of the diagonal subproblems.
The general
convergence theorem and the model-specific update estimates therefore give
the following result.
\begin{corollary}[Convergence for the direct-fit model]\label{corollary:direct_fit}
Fix $\lambda>0$ in the direct-fit model \eqref{eq:cost_direct}.
For every initial point
$(\matr{A}_0,\matr{D}_0)\in\Delta_n^m\times(\set{D}_m)^L$,
the sequence $\{(\matr{A}_k,\matr{D}_k)\}$ generated by
\cref{alg:gradient-based-direct} is bounded, admits an accumulation point,
and converges to a limit $(\matr{A}^*,\matr{D}^*)$.
If, in addition, \cref{assump:interiority} holds for the $\matr{A}$-block,
then the limit is a constrained stationary point of $f_{\mathrm{dir}}$, and
the convergence rate in \eqref{eq:conver_rate} holds with the constrained
\L{}ojasiewicz exponent of $f_{\mathrm{dir}}$ at that limit.
\end{corollary}

For the indirect-fit model, compactness of the simplex guarantees the existence of an accumulation point. The general convergence theorem and the model-specific update estimates therefore give the following result.
\begin{corollary}[Convergence for the indirect-fit model]
\label{corollary:indirect_fit}
Consider the indirect-fit model \eqref{eq:cost_indirect} under the square
setting $n=m$, where each $\matr{C}^{(\ell)}$ is nonsingular.
For every initial point $\matr{A}_0\in\Delta_n^m$, the sequence
$\{\matr{A}_k\}$ generated by \cref{alg:gradient-based-indirect}
converges to a limit $\matr{A}^*$. If, in addition,
\cref{assump:interiority} holds, then $\matr{A}^*$ is a constrained
stationary point of $f_{\mathrm{indir}}$, and the convergence rate in
\eqref{eq:conver_rate} holds with the constrained \L{}ojasiewicz exponent
of $f_{\mathrm{indir}}$ at that limit.
\end{corollary}

\section{Solving the direct-fit model}\label{sec:semi-nnica-direct}

For the direct-fit model \eqref{eq:cost_direct} with fixed $\lambda>0$, we apply the general
framework \eqref{eq:cost_general} with $f=f_{\mathrm{dir}}$, $\matr{U}=\matr{A}\in\Delta_n^m$,
and
$\matr{V}=\matr{D}=(\matr{D}^{(1)},\ldots,\matr{D}^{(L)})
\in\Omega=(\set{D}_m)^L$. 

We omit the range $\ell=1,\ldots,L$
from collections of diagonal factors when it is clear.
Throughout this section and the direct-fit update proofs, we use the restricted functions 
$f_{\matr{D}}:\matr{A}\mapsto f_{\mathrm{dir}}(\matr{A},\matr{D})$ and
$f_{\matr{A}}:\matr{D}\mapsto f_{\mathrm{dir}}(\matr{A},\matr{D})$.
The space $\Omega$, equipped with the product Frobenius inner product,
is a Euclidean space identified with $\mathbb{R}^{mL}$ through the
diagonal entries. Gradients with respect to $\matr{D}$ are taken in
this space. If $\nabla^{\mathrm{amb}}$ denotes differentiation with
respect to unrestricted matrix entries, then we have
$\nabla f_{\matr{A}}\!\left(\matr{D}\right)
=\left(
\operatorname{Diag}\!\left(\operatorname{diag}\!\left(
\nabla^{\mathrm{amb}}_{\matr{D}^{(\ell)}}f_{\mathrm{dir}}(\matr{A},\matr{D})
\right)\right)
\right)_{\ell=1}^{L}.$
In particular, we have
$\nabla f_{\matr{A}}\!\left(\matr{D}\right)
=\nabla_{\matr{D}}\tilde{f}_{\mathrm{dir}}(\matr{A},\matr{D})
+2\lambda\matr{D}$.
The block-selection rule below uses this regularized gradient. With $\matr{D}_{k-1}$ fixed, for a candidate column $p_k$
and pair $(i_k,j_k)$ in the $\matr{A}$-update,
define the corresponding elementary function by
\begin{equation}\label{eq:hk_direct}
h_k(\alpha)
\eqdef
f_{\mathrm{dir}}\!\left(
\matr{A}_{k-1}
+\alpha(\matr{A}_{k-1})_{j_k,p_k}
(\vect{e}_{i_k}-\vect{e}_{j_k})
\vect{e}_{p_k}^{\T},
\matr{D}_{k-1}
\right).
\end{equation}
The function $h_k(\alpha)$ has a quartic polynomial form,
as shown in \eqref{eq:h_quartic}, with coefficients
specified immediately thereafter.
The resulting algorithm is given in
\cref{alg:gradient-based-direct}.
% Since the penalty is independent of $\matr{A}$, the $\matr{A}$-gradient,
% the pair-selection criterion, and the ECT objective differences are unchanged.

\begin{subalgorithm}[H]
\caption{BCD-SubG for the direct-fit model.}
\label{alg:gradient-based-direct}
\KwIn{Problem \eqref{eq:cost_direct} with fixed $\lambda>0$, a parameter
$0<\delta\leq1/[n\sqrt{m(n-1)}]$, and a starting point
$\matr{A}_0\in\Delta_n^m$ and
$\matr{D}_0=(\matr{D}^{(\ell)}_0)_{\ell=1}^{L}\in(\set{D}_m)^L$.}
\KwOut{Sequence of iterates $\{(\matr{A}_k,\matr{D}_k)\}$.}
\begin{itemize}
\item \textbf{For} $k = 1, 2, \dots$ [until a stopping criterion is satisfied] do
\item If
\[
\left\|\nabla f_{\matr{A}_{k-1}}\!\left(\matr{D}_{k-1}\right)\right\|
\;\geq\;
\operatorname{dist}\!\left(\matr{0},\,
\partial_{\Delta_n^m}f_{\matr{D}_{k-1}}(\matr{A}_{k-1})\right),
\]
we set $\matr{A}_{k}=\matr{A}_{k-1}$ and update $\matr{D}_k$ over
$(\set{D}_m)^L$.
\begin{itemize} 
\item[$\bullet$] Update $\matr{D}_k$ by solving
\[
\matr{D}_k
\in
\argmin_{\matr{D}\in(\mathcal{D}_m)^L}
f_{\matr{A}_{k-1}}(\matr{D}).
\]
\end{itemize}
\item Otherwise, we set $\matr{D}_{k}=\matr{D}_{k-1}$ and perform an
update over $\Delta_n^m$.
\begin{itemize}
\item[$\bullet$] Choose $p_k\in\{1,\ldots,m\}$ and an index pair
$(i_k,j_k)$ with $i_k\neq j_k$ and
$(\matr{A}_{k-1})_{j_k,p_k}>0$ such that
\begin{align*}
\operatorname{dist}\!\left(0,\,\partial h_{k}(0)\right)\geq \delta
\operatorname{dist}\!\left(\matr{0},\,
\partial_{\Delta_n^m}f_{\matr{D}_{k-1}}(\matr{A}_{k-1})\right),
\end{align*}
where $h_k$ is defined in \eqref{eq:hk_direct}.
\item[$\bullet$] Compute $\alpha_k^*$ by applying
\hyperlink{rule:R2}{Rule~(R2)} to $h_k$.
\item[$\bullet$] Set $(\matr{A}_k)_{:,q}=(\matr{A}_{k-1})_{:,q}$ for all $
q\neq p_k$ and update $(\matr{A}_k)_{:,p_k}
=
\matr{T}_{i_k,j_k}(\alpha_k^*)
(\matr{A}_{k-1})_{:,p_k}.$
\end{itemize}
\item \textbf{End for}
\end{itemize}
\end{subalgorithm}

%\subsection{The $\matr{A}$-update.}

With $\matr{A}_{k-1}$ fixed, the $\matr{D}$-update solves $\matr{D}_k\in
\operatorname*{arg\,min}_{\matr{D}\in(\set{D}_m)^L}
f_{\matr{A}_{k-1}}(\matr{D})$.
This subproblem is strongly convex and has a unique minimizer  
\begin{equation}\label{eq:D-update-1}
\matr{D}_k^{(\ell)}
=
\operatorname{Diag}\left(
\bigl(
\matr{B}_{k-1}^{\T}\matr{B}_{k-1}
+\lambda\matr{I}_m
\bigr)^{-1}
\matr{B}_{k-1}^{\T}
\vecl\bigl(\matr{C}^{(\ell)}\bigr)
\right),
\quad \ell=1,\ldots,L,
\end{equation}
where $\matr{B}_{k-1}\eqdef\matr{A}_{k-1}\kr\matr{A}_{k-1}$. The inverse in \eqref{eq:D-update-1} is well defined since
$\matr{B}_{k-1}^{\T}\matr{B}_{k-1}+\lambda\matr{I}_m
\succeq\lambda\matr{I}_m\succ 0$. 
The $\matr{A}$-update applies \hyperlink{rule:R2}{Rule~(R2)} to $h_k$,
selecting the first local minimizer encountered from $\alpha=0$ along
the descent direction.

We next establish the model-specific properties required for
\cref{corollary:direct_fit}. We first show that the sequence generated by
\cref{alg:gradient-based-direct} is contained in a compact subset of
the feasible set. Since $\Delta_n^m$ is compact, it suffices to prove
that the tuple iterates $\matr{D}_k$ are uniformly bounded. We then
establish the descent, zero-step, and step-length estimates for the
$\matr{A}$- and $\matr{D}$-updates in
\cref{lem:A-update-direct,lem:D-update-direct}, respectively.

\begin{proposition}[Compactness of the direct-fit iterates]\label{prop:boundedness-direct}
For fixed $\lambda>0$ and every feasible initial point, the sequence
$\{(\matr{A}_k,\matr{D}_k)\}$ generated by
\cref{alg:gradient-based-direct} remains in a compact sublevel set of
$f_{\mathrm{dir}}$. In particular, the diagonal iterates are uniformly bounded,
and the full sequence admits an accumulation point.
\end{proposition}

\begin{proof}
Let $f_0\eqdef f_{\mathrm{dir}}(\matr{A}_0,\matr{D}_0)$.
An exact $\matr{D}$-update minimizes the regularized objective with
$\matr{A}$ fixed. An $\matr{A}$-update follows \hyperlink{rule:R2}{Rule~(R2)}, which does not
increase the objective; the penalty is constant during this update.
Thus, for every $k\geq0$, we have
\[
\lambda\sum_{\ell=1}^{L}\|\matr{D}_k^{(\ell)}\|^2
\leq f_{\mathrm{dir}}(\matr{A}_k,\matr{D}_k)
\leq f_{\mathrm{dir}}(\matr{A}_0,\matr{D}_0)=f_0.
\]
Hence, we obtain $\|\matr{D}_k\|\leq\sqrt{f_0/\lambda}$ for every $k$.
The sublevel set
$\{(\matr{A},\matr{D})\in\Delta_n^m\times(\set{D}_m)^L:
f_{\mathrm{dir}}(\matr{A},\matr{D})\leq f_0\}$
is closed by continuity and bounded by the same estimate and compactness
of $\Delta_n^m$. It is therefore compact and contains the full sequence,
which consequently admits an accumulation point.
\end{proof}

\begin{lemma}[The $\matr{A}$-update]\label{lem:A-update-direct}
Whenever the $\matr{A}$-update is performed in
\cref{alg:gradient-based-direct}, it satisfies the following properties, with $\eta_{\mathrm{A},1}=1/6$:
\begin{itemize}
\item[(i)] For all $k$, we have 
\[
h_k(0)-h_k(\alpha_k^*)
\geq\eta_{\mathrm{A},1}\operatorname{dist}\!\left(
0,\partial h_k(0)\right)|\alpha_k^*|.
\]
\item[(ii)] If $\operatorname{dist}\!\left(
0,\partial h_k(0)
\right)=0$, then $\alpha_k^*=0$.
\item[(iii)] If \cref{assump:interiority}
holds for the $\matr{A}$-block, then there exists a constant $K_{\mathrm{A},1}>0$, independent of $k$, such that, with
$\kappa_{\mathrm{A},1}\eqdef1/K_{\mathrm{A},1}$, for
all sufficiently large $k$, we have
\[
|\alpha_k^*|
\geq \kappa_{\mathrm{A},1}\,\operatorname{dist}\!\left(
0,\partial h_k(0)
\right).
\]
\end{itemize}
\end{lemma}

The proof is given in \cref{app:proof-direct-a}.

%\subsection{The $\matr{D}$-update.}

\begin{lemma}[The $\matr{D}$-update]\label{lem:D-update-direct}
Fix $\lambda>0$ and let $K_{\mathrm{D}}=2(m+\lambda)$,
$\eta_{\mathrm{D}}\eqdef\lambda/K_{\mathrm{D}}$, and
$\kappa_{\mathrm{D}}\eqdef1/K_{\mathrm{D}}$.
At every iteration $k$ at which the $\matr{D}$-update is performed,
the following properties hold:
\begin{itemize} 
\item[(i)] The following sufficient-descent estimate holds: 
\[
\begin{aligned}
f_{\matr{A}_{k-1}}(\matr{D}_{k-1})
-f_{\matr{A}_{k-1}}(\matr{D}_{k})
\geq\eta_{\mathrm{D}}
\left\|\nabla f_{\matr{A}_{k-1}}\!\left(\matr{D}_{k-1}\right)\right\|
\|\matr{D}_{k}-\matr{D}_{k-1}\|.
\end{aligned}
\]
\item[(ii)] If
$\left\|\nabla f_{\matr{A}_{k-1}}\!\left(\matr{D}_{k-1}\right)\right\|=0$,
then $\matr{D}_{k}=\matr{D}_{k-1}$.
\item[(iii)] The step length satisfies the following lower bound:
\[
\|\matr{D}_{k}-\matr{D}_{k-1}\|
\geq\kappa_{\mathrm{D}}
\left\|\nabla f_{\matr{A}_{k-1}}\!\left(\matr{D}_{k-1}\right)\right\|.
\]
\end{itemize}
\end{lemma}

The proof is given in \cref{app:proof-direct-a}.

\section{Solving the indirect-fit model}\label{sec:semi-nnica}

For the indirect-fit model \eqref{eq:cost_indirect}, we apply the general framework
\eqref{eq:cost_general} with $f=f_{\mathrm{indir}}$.
Since the objective involves
$\bigl(\matr{C}^{(\ell)}\bigr)^{-1}$, we assume that each
$\matr{C}^{(\ell)}$ is nonsingular and restrict the indirect-fit
formulation to the square case $n=m$.
In this setting, set $\matr{U}=\matr{A}\in\Delta_n^m$ and represent the
absent second block by $\matr{V}=\matr{0}_{1\times1}\in
\Omega=\{\matr{0}_{1\times1}\}\subseteq\mathbb{R}^{1\times1}$.
For a candidate column $p_k$ and pair $(i_k,j_k)$,
define the corresponding elementary function by
\begin{equation}\label{eq:hk_indirect}
h_k(\alpha)
\eqdef
f_{\mathrm{indir}}\!\left(
\matr{A}_{k-1}
+\alpha(\matr{A}_{k-1})_{j_k,p_k}
(\vect{e}_{i_k}-\vect{e}_{j_k})
\vect{e}_{p_k}^{\T}
\right).
\end{equation}
The function $h_k(\alpha)$ has a quadratic polynomial form,
as shown in \eqref{eq:h_quadratic}, with coefficients
specified immediately thereafter.
The resulting algorithm is given in
\cref{alg:gradient-based-indirect}. 

\begin{subalgorithm}[H]
\caption{BCD-SubG for the indirect-fit model.}
\label{alg:gradient-based-indirect}
\KwIn{Problem \eqref{eq:cost_indirect}, a parameter $0 < \delta \leq 1/[n\sqrt{m(n-1)}]$, and a starting point $\matr{A}_{0}\in\Delta_{n}^{m}$.}
\KwOut{Sequence of iterations $\{\matr{A}_{k}\}$.}
\begin{itemize}
\item \textbf{For} $k = 1, 2, \dots$ [until a stopping criterion is satisfied] do 
\item[$\bullet$] Choose $p_k\in\{1,\ldots,m\}$ and an index pair
$(i_k,j_k)$ with $i_k\neq j_k$ and
$(\matr{A}_{k-1})_{j_k,p_k}>0$ such that
\begin{align*}
\operatorname{dist}\!\left(0,\,\partial h_{k}(0)\right)\;\geq\; \delta
\operatorname{dist}\!\left(\matr{0},\,\partial_{\Delta_n^m}f_{\mathrm{indir}}(\matr{A}_{k-1})\right),
\end{align*}
where $h_k$ is defined in \eqref{eq:hk_indirect}.
\item[$\bullet$] Compute $\alpha_k^*$ by applying
\hyperlink{rule:R1}{Rule~(R1)} to $h_k$.
\item[$\bullet$] Set $(\matr{A}_k)_{:,q}=(\matr{A}_{k-1})_{:,q}$ for all $q \neq p_k$ and update $(\matr{A}_k)_{:,p_k}=\matr{T}_{i_k,j_k}(\alpha_k^*)(\matr{A}_{k-1})_{:,p_k}$.
\item \textbf{End for}
\end{itemize}
\end{subalgorithm}

The $\matr{A}$-update applies \hyperlink{rule:R1}{Rule~(R1)} to $h_k$, so $\alpha_k^*$ is a global minimizer of $h_k$ over the feasible interval.

We next establish the model-specific properties required for
\cref{corollary:indirect_fit}. Since $\Delta_n^m$ is compact, the sequence
generated by \cref{alg:gradient-based-indirect} is contained in the compact
feasible set and therefore admits an accumulation point. We then establish
the descent, zero-step, and step-length estimates for the $\matr{A}$-update
in \cref{lem:A-update-indirect}.

\begin{lemma}[The $\matr{A}$-update]\label{lem:A-update-indirect}
In \cref{alg:gradient-based-indirect}, the update of $\matr{A}_k$ satisfies the following properties, with $\eta_{\mathrm{A},2}=1/2$:
\begin{itemize}
\item[(i)] For all $k$, we have
\[
h_k(0) - h_k(\alpha_k^*) \geq\eta_{\mathrm{A},2}\operatorname{dist}\!\left(
0,\partial h_k(0)
\right)|\alpha_k^*|.
\]
\item[(ii)] If $\operatorname{dist}\!\left(
0,\partial h_k(0)
\right)=0$, then $\alpha_k^*=0$.
\item[(iii)] If \cref{assump:interiority} holds for the $\matr{A}$-block,
then there exists a constant $K_{\mathrm{A},2}>0$, independent of $k$, such that, with
$\kappa_{\mathrm{A},2}\eqdef1/K_{\mathrm{A},2}$, for
all sufficiently large $k$, we have
\[
|\alpha_k^*|
\geq \kappa_{\mathrm{A},2}\,\operatorname{dist}\!\left(
0,\partial h_k(0)
\right).
\]
\end{itemize}
\end{lemma}

The proof is given in \cref{app:proof-indirect-a}.

%\section{Parallel Computation Strategy}

\section{Numerical experiments}
\label{sec:numerical_experiments}

We evaluate BCD-SubG on synthetic mixtures, simulated magnetic resonance
spectroscopy (MRS) data, and joint diagonalization problems. The first two
experiments use the regularized direct-fit formulation~\eqref{eq:cost_direct}
and compare recovery accuracy and wall-clock time. The third uses the
indirect-fit formulation~\eqref{eq:cost_indirect} and compares final
objective values and wall-clock time.

\subsection{Experimental setup}

In~\eqref{Ica_model}, let
$\matr{S}\eqdef[\vect{s}(1),\ldots,\vect{s}(T)]$,
$\matr{E}\eqdef[\vect{\varepsilon}(1),\ldots,\vect{\varepsilon}(T)]$,
and $\matr{X}\eqdef[\vect{x}(1),\ldots,\vect{x}(T)]$
denote the source, noise, and observation matrices, respectively.
For mixture experiments, we define
$\operatorname{SNR}
\eqdef
20\log_{10}\bigl(
\|\matr{A}\matr{S}\|
/
\|\matr{E}\|
\bigr).$
Given an estimate $\widehat{\matr{A}}$, we recover the sources by
least squares: $\widehat{\matr{S}}=\widehat{\matr{A}}^{\dagger}\matr{X}$.
Let $\vect{s}_i^{\T}$ and $\widehat{\vect{s}}_i^{\T}$ be the $i$th
rows of $\matr{S}$ and $\widehat{\matr{S}}$, and let $\vect{a}_i$ and
$\widehat{\vect{a}}_i$ be the $i$th columns of $\matr{A}$ and
$\widehat{\matr{A}}$. We use scale-invariant matching errors,
following~\cite[Eq.~(18)]{zarzoso2009robust}:
\begin{align*}
\operatorname{err}_{\mathrm S}
\eqdef
\frac{1}{m}\sum_{i=1}^{m}
\left(
1-
\frac{
|\vect{s}_{i}^{\T}\widehat{\vect{s}}_{\pi_{\mathrm S}(i)}|^2
}{
\|\vect{s}_{i}\|^2
\|\widehat{\vect{s}}_{\pi_{\mathrm S}(i)}\|^2
}
\right),\ \
\operatorname{err}_{\mathrm A}
\eqdef
\frac{1}{m}\sum_{i=1}^{m}
\left(
1-
\frac{
|\vect{a}_{i}^{\T}\widehat{\vect{a}}_{\pi_{\mathrm A}(i)}|^2
}{
\|\vect{a}_{i}\|^2
\|\widehat{\vect{a}}_{\pi_{\mathrm A}(i)}\|^2
}
\right).
\end{align*}
The permutations $\pi_{\mathrm S},\pi_{\mathrm A}\in\set{S}_m$,
where $\set{S}_m$ is the set of permutations of $\{1,\ldots,m\}$,
are computed independently by greedy matching: each step selects
the unmatched pair with the largest squared normalized inner product.
Thus, the metrics assess source and mixing-matrix recovery separately.

At each parameter setting, all methods use the same 200 independent
instances. We report median errors or final objective values and
median end-to-end wall-clock times. Following~\cite{wang2014nonnegative},
each iterative method stops when
\[
\frac{|f_k-f_{k-1}|}
{\max\{|f_{k-1}|,\varepsilon_{\mathrm{den}}\}}
\leq 10^{-4},
\]
where $f_k$ is its objective value at iteration $k$ and
$\varepsilon_{\mathrm{den}}>0$ prevents division by a numerically zero value.

We test two BCD-SubG variants. BCD-SubG-Max selects an admissible
pair with a large one-dimensional stationarity measure;
BCD-SubG-Cyclic selects the first admissible pair in a fixed cyclic
scan. Both use the same ECT update. In Experiments~1 and~2, we set
$\lambda=10^{-6}$ and solve the regularized diagonal subproblem exactly
by~\eqref{eq:D-update-1} after a fixed number of $\matr{A}$-update phases.
The stopping test uses $f_k=f_{\mathrm{dir}}(\matr{A}_k,\matr{D}_k)$.
These are practical variants: the convergence theorem covers the
adaptive block rule in~\cref{alg:gradient-based-direct}, not this fixed schedule.

\subsection{Experiment 1: synthetic data}
\label{sec:numerical_experiments-synthetic}

We generate mixtures using~\eqref{Ica_model}, with two independent
sources drawn from Laplace and uniform distributions, respectively,
then centered and scaled to unit second moment.
From two random simplex vectors $\widetilde{\vect{a}}_1$ and
$\widetilde{\vect{a}}_2$, we set
$\vect{a}_1=\widetilde{\vect{a}}_1$ and
$\vect{a}_2=(1-\beta_{\mathrm{coh}})\widetilde{\vect{a}}_1
+\beta_{\mathrm{coh}}\widetilde{\vect{a}}_2$.
Smaller $\beta_{\mathrm{coh}}$ gives more coherent columns.
We generate Gaussian noise with covariance
\begin{align*}
\matr{\Sigma}_{\rho_{\mathrm{n}},\varphi}
=(1-\rho_{\mathrm{n}})\matr{I}_n+\rho_{\mathrm{n}}\matr{R}_{\varphi},
\ \ 
\matr{R}_{\varphi}
=\operatorname{Re}\!\left(
\vect{v}(\varphi)\vect{v}(\varphi)^{\herm}
\right),\ \
\vect{v}(\varphi)
=\left[
1,e^{\mathrm{i}\pi\sin\varphi},\ldots,
e^{\mathrm{i}(n-1)\pi\sin\varphi}
\right]^{\T}.
\end{align*}
Here, $\rho_{\mathrm{n}}\in[0,1)$ controls cross-channel correlation
and $\varphi$ is the interference angle. This model is adapted
from~\cite[Eq.~(2)]{qian2016mainlobe}. We scale the noise to the prescribed SNR.

We compare both variants with the Semi-NNICA methods
$\mathrm{JDC}^{+}_{\mathrm{LU-3}}$ and
$\mathrm{JDC}^{+}_{\mathrm{ADMM-3}}$~\cite{wang2014nonnegative}.
The defaults are $n=10$, $m=2$, $T=1000$,
$\operatorname{SNR}=10\,\mathrm{dB}$, $\rho_{\mathrm{n}}=0.2$,
$\beta_{\mathrm{coh}}=1$, and $\varphi=90^\circ$.
Keeping the other parameters fixed, we vary
$\operatorname{SNR}$,
$T$ or
$\beta_{\mathrm{coh}}$.
As shown in~\cref{fig:synthetic-results}, both BCD-SubG variants
generally achieve lower median recovery errors and shorter wall-clock
time than the two $\mathrm{JDC}^{+}$ methods across the tested settings.
Between the proposed variants, BCD-SubG-Cyclic often gives lower
recovery errors, whereas BCD-SubG-Max is often faster.

\begin{figure}[!t]
\centering
\includegraphics[width=\textwidth]{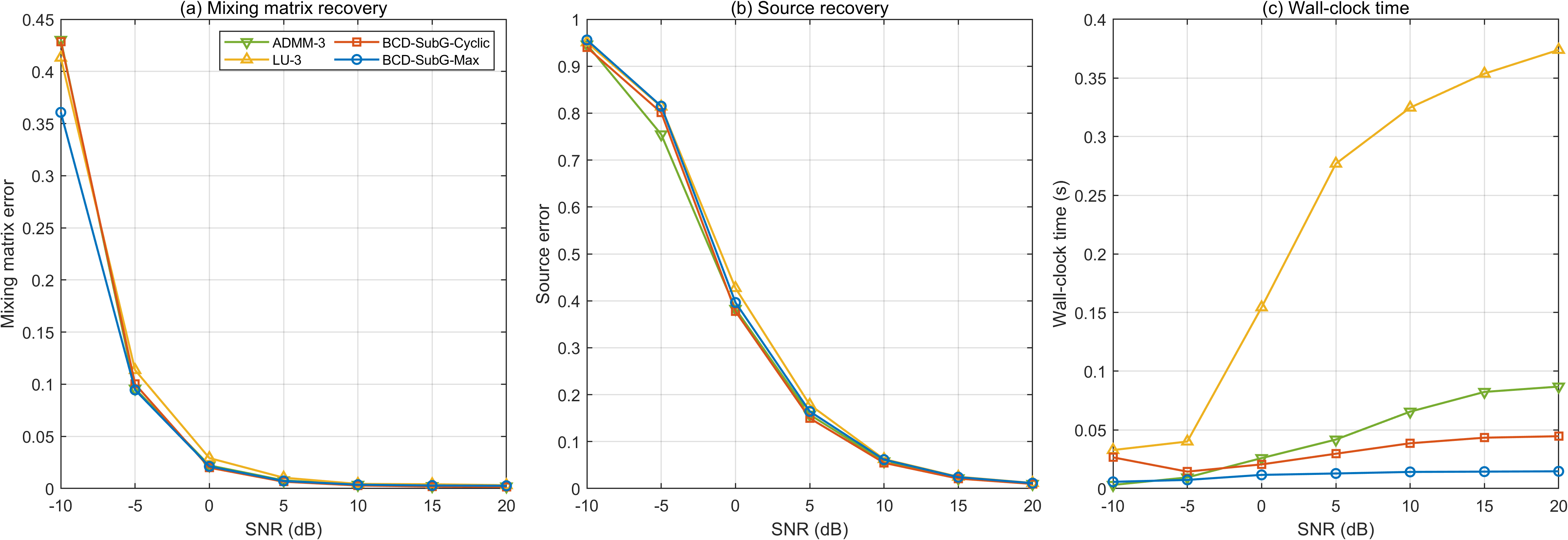}
\vspace{0.4em}

\includegraphics[width=\textwidth]{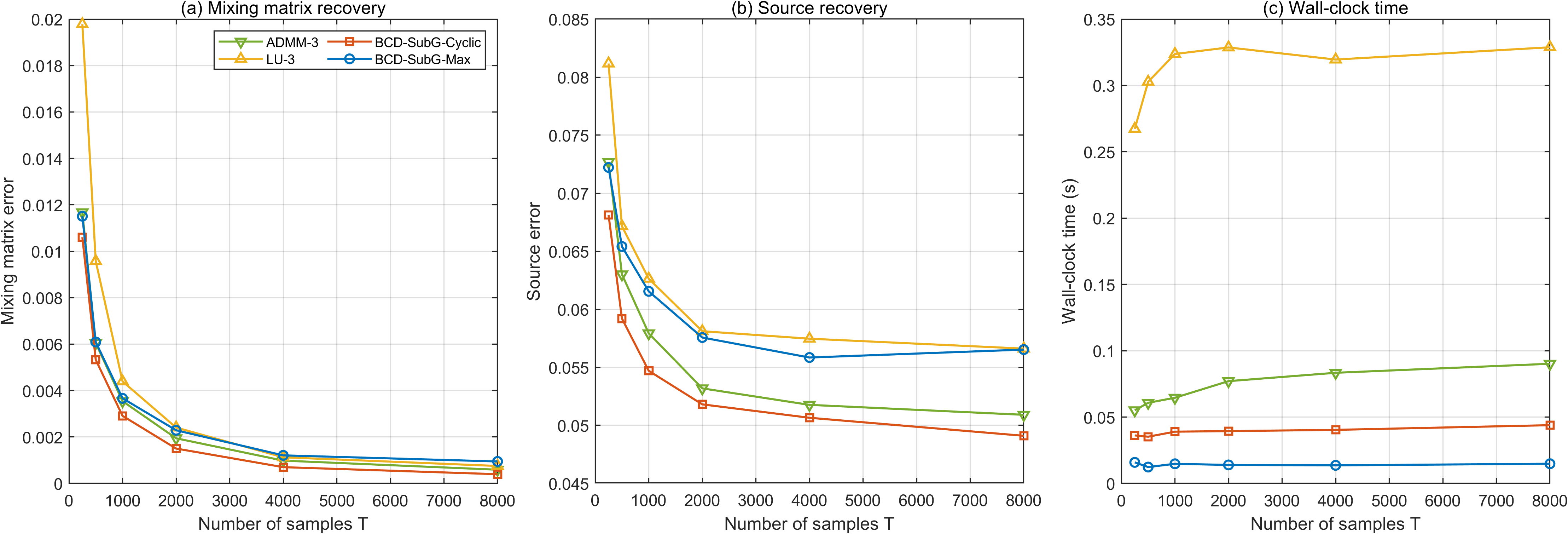}
\vspace{0.4em}

\includegraphics[width=\textwidth]{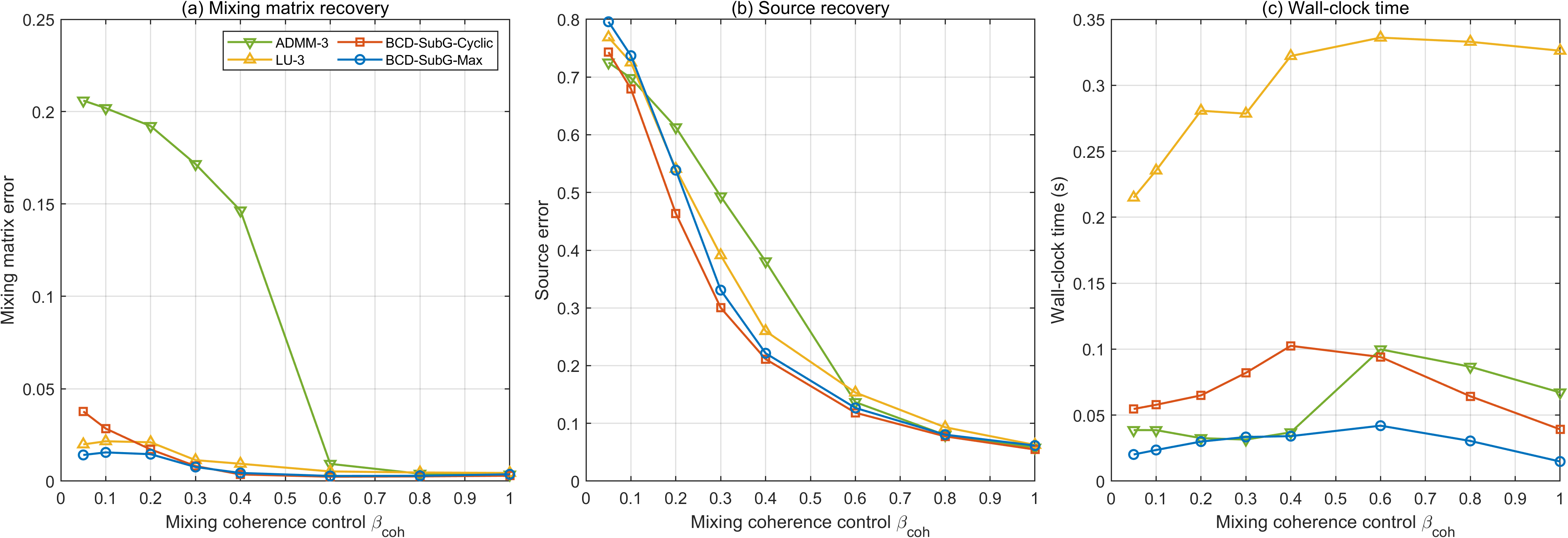}
\caption{Median mixing-matrix error (left), source error (center), and
wall-clock time (right) over 200 paired synthetic instances. Rows vary
SNR, $T$, and $\beta_{\mathrm{coh}}$, respectively; other parameters
remain at their default values. Lower is better.}
\label{fig:synthetic-results}
\end{figure}

\subsection{Experiment 2: simulated MRS data}
\label{sec:numerical_experiments-MRS}

MRS signals combine metabolite spectra with nonnegative mixing
coefficients representing concentrations~\cite{wang2014nonnegative,coloigner2014semi}.
We simulate choline and myo-inositol spectra using Lorentzian and
Gaussian peaks with prescribed locations and widths, then normalize
each spectrum to unit maximum amplitude. 
For each instance, we draw mixing-matrix entries independently from
$\operatorname{Unif}[0,1]$ and normalize each column to sum to one.
This normalization preserves the unregularized direct-fit representation
through rescaling of the diagonal factors. The quadratic penalty
in~\eqref{eq:cost_direct}, however, is imposed after normalization
and is not scale invariant.
We add independent zero-mean Gaussian noise at the prescribed SNR
and compare both variants with the two $\mathrm{JDC}^{+}$ methods.
The defaults are $n=5$, $m=2$, $T=1000$, and
$\operatorname{SNR}=10\,\mathrm{dB}$.
We vary $\operatorname{SNR}$
or $n$, keeping the other parameters fixed.

\Cref{fig:mrs-results} shows the same general pattern as the
synthetic experiments: both proposed variants tend to give lower
median recovery errors and shorter wall-clock time than the baselines.
BCD-SubG-Cyclic often favors recovery accuracy, while BCD-SubG-Max
often favors speed. 

\begin{figure}[!t]
\centering
\includegraphics[width=\textwidth]{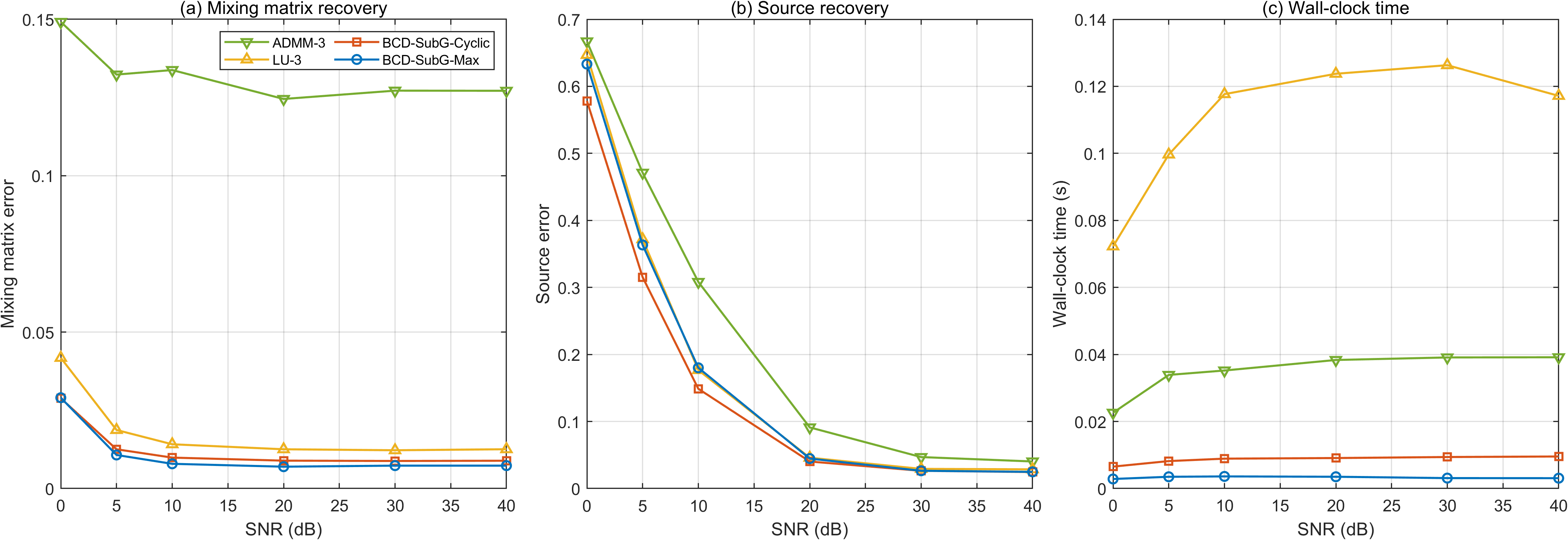}
\vspace{0.5em}

\includegraphics[width=\textwidth]{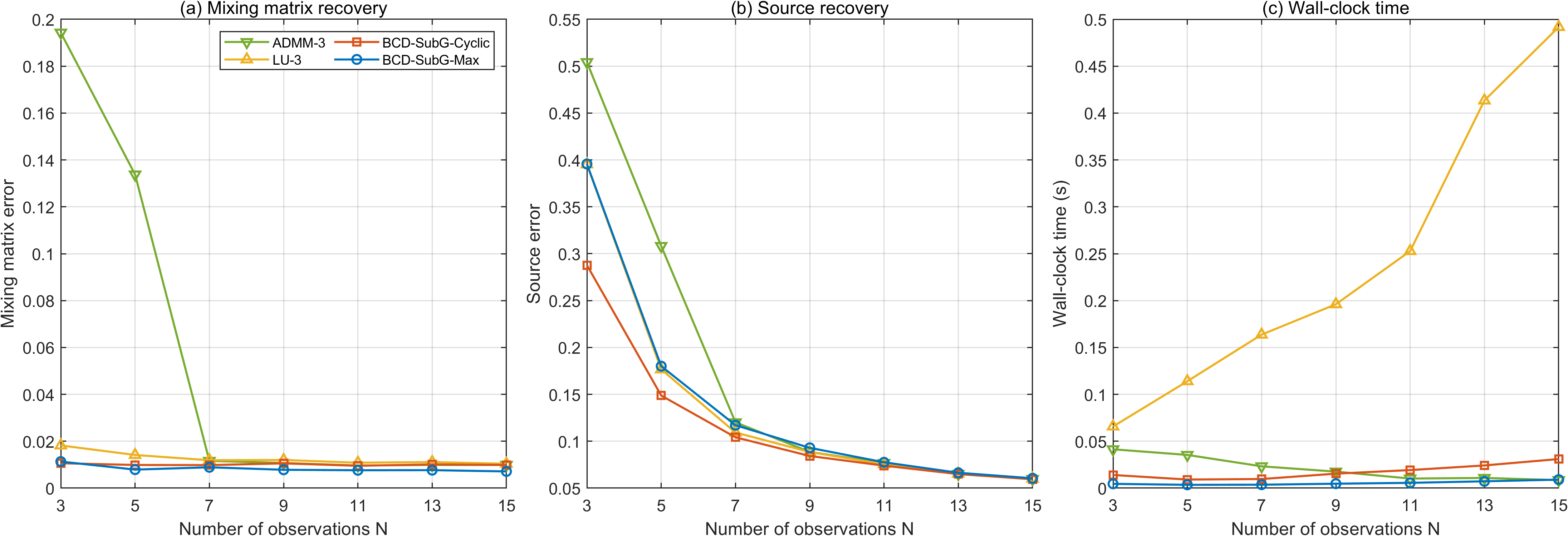}
\caption{Median mixing-matrix error (left), source error (center), and
wall-clock time (right) over 200 paired simulated MRS instances.
The top row varies SNR with $n=5$; the bottom row varies $n$ with
$\operatorname{SNR}=10\,\mathrm{dB}$. Lower is better.}
\label{fig:mrs-results}
\end{figure}

\subsection{Experiment 3: joint diagonalization}
\label{sec:numerical_experiments-indirect}

We compare both variants with first-order methods for the
simplex-constrained indirect-fit model~\eqref{eq:cost_indirect}.
For $n=m$, we generate data from the joint diagonalization
model~\cite{coloigner2014semi}, given by
$$\matr{C}^{(\ell)}
=\matr{A}\matr{D}^{(\ell)}\matr{A}^{\T}
+\matr{E}^{(\ell)},\ \ \ \ell=1,\ldots,L.$$
We sample the diagonal entries of $\matr{D}^{(\ell)}$ from
$\operatorname{Unif}[0.5,2]$ and generate random
$\matr{A}\in\Delta_n^n$, retaining only well-conditioned matrices.
The perturbations $\matr{E}^{(\ell)}$ are random symmetric matrices
scaled to keep all observed matrices positive definite.
Here, we measure the noise level by
$$\operatorname{SNR}\eqdef20\log_{10}
\bigl(
\|(\matr{A}\matr{D}^{(\ell)}\matr{A}^{\T})_{\ell=1}^{L}\|
/
\|(\matr{E}^{(\ell)})_{\ell=1}^{L}\|
\bigr).$$
All methods use the same feasible initial matrix
\[
\matr{A}_0
\eqdef
(1-\rho_{\mathrm{init}})\frac{\vect{1}\vect{1}^{\T}}{n}
+\rho_{\mathrm{init}}\matr{A}_{\mathrm{rand}},
\]
where the columns of $\matr{A}_{\mathrm{rand}}$ are sampled independently
from $\operatorname{Dirichlet}(0.15\vect{1})$.
We set $n=m=4$, $L=3$, and $\rho_{\mathrm{init}}=0.2$, and vary
$\operatorname{SNR}\in\{-10,-5,0,5,10,20\}\,\mathrm{dB}$.
We use the same stopping criterion as in the preceding experiments.
The baselines are Block-FW~\cite{lacostejulien2013block},
active-set FW (AS-FW)~\cite{cristofari2020active}, and
block-coordinate projected gradient
(Block-PG)~\cite{birgin2000nonmonotone}.
All three select columns cyclically and apply their respective update rules.

\begin{figure}[!t]
\centering
\includegraphics[width=0.49\textwidth]
    {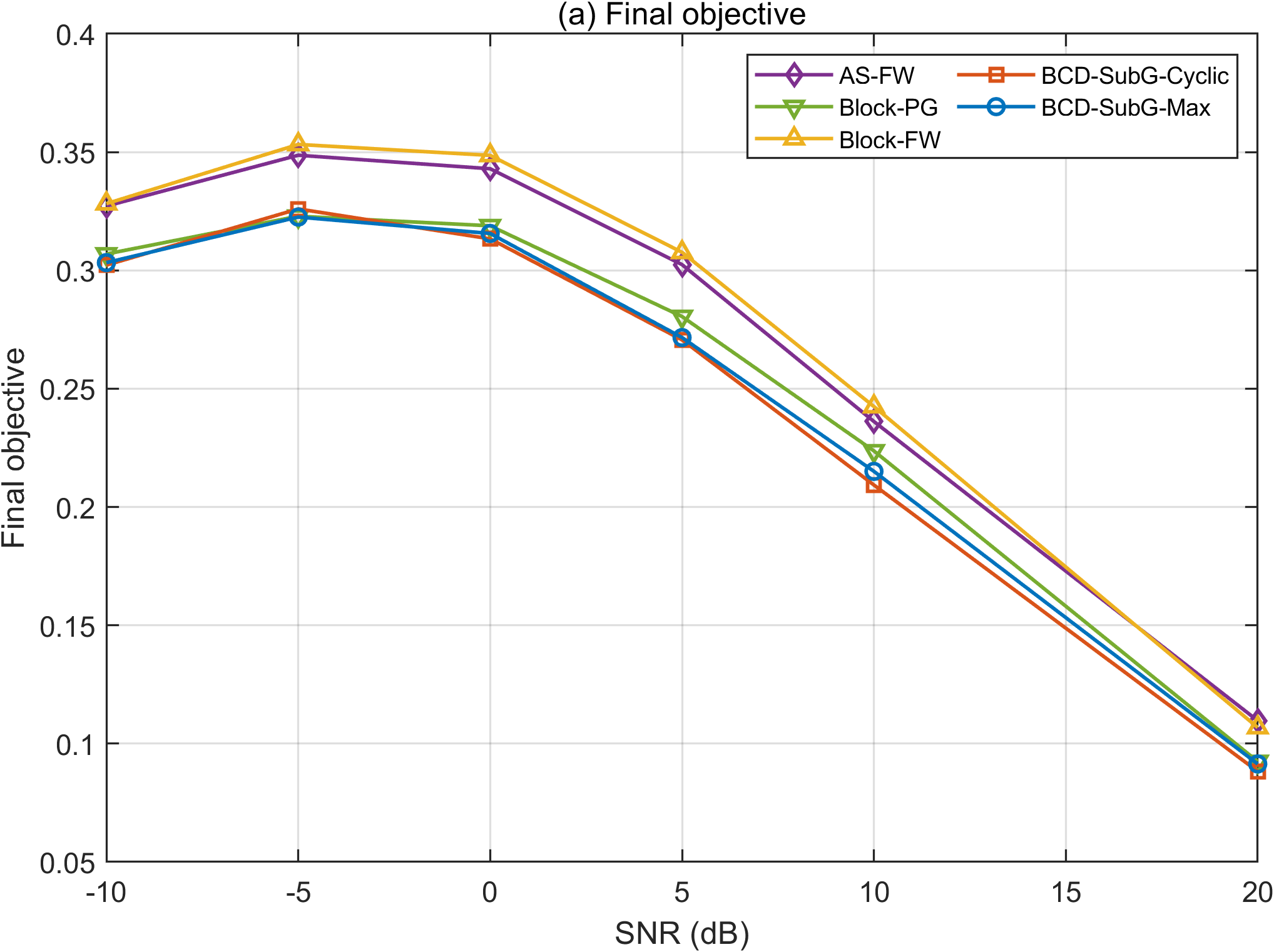}
\hfill
\includegraphics[width=0.49\textwidth]
    {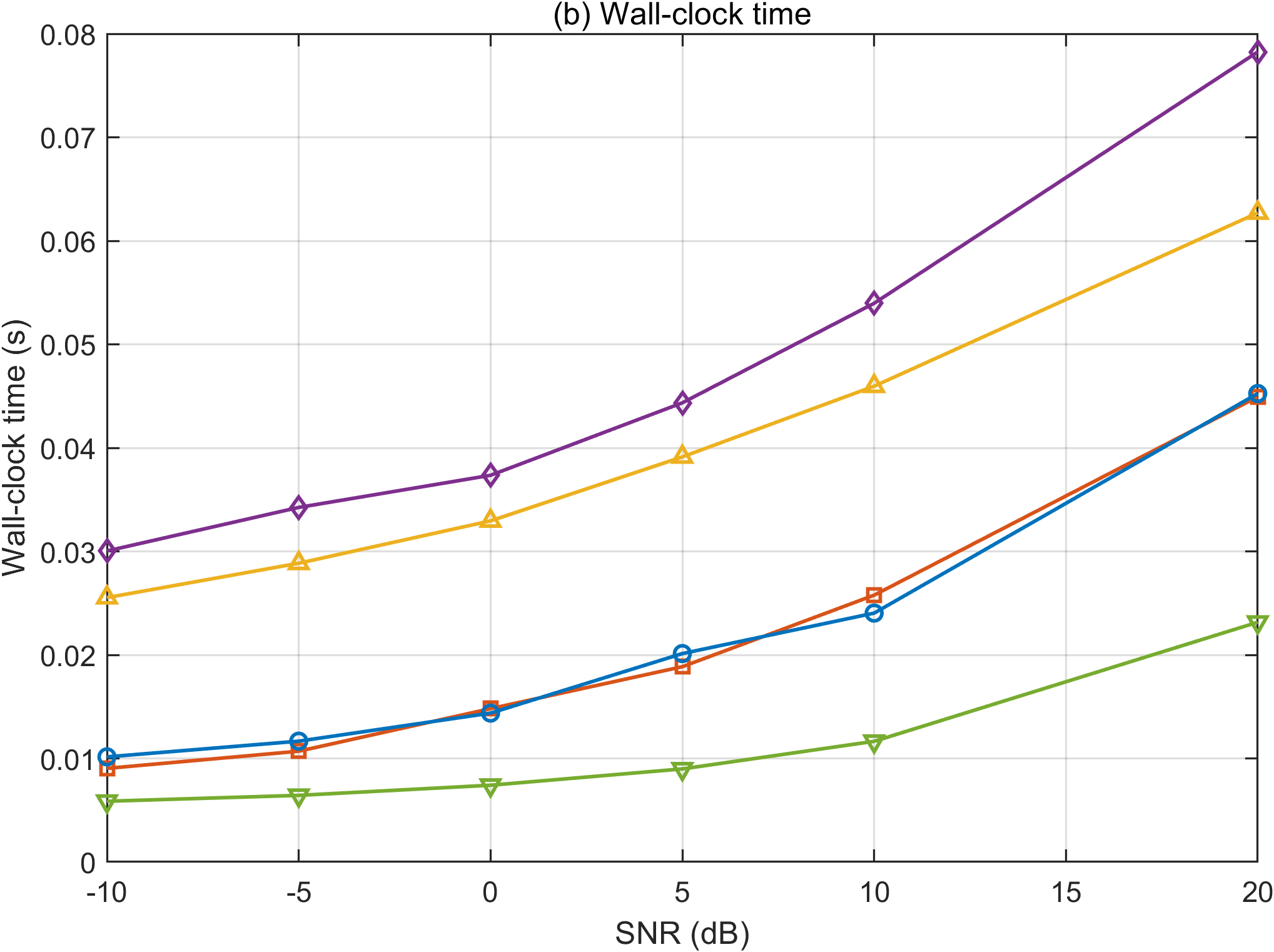}
\caption{Median final objective value (left) and wall-clock time
(right) over 200 paired trials at each SNR level. Lower is better.}
\label{fig:indirect-fit-comparison}
\end{figure}

As shown in~\cref{fig:indirect-fit-comparison}, at each tested SNR,
one of the two BCD-SubG variants achieves the lowest median final
objective value. Block-PG is consistently the fastest method,
while both BCD-SubG variants are faster than Block-FW and AS-FW.
These results are consistent with the example
in~\cref{eq:simplex-motivation}: Block-PG can be faster, whereas
the BCD-SubG variants can reach points with lower objective values.

\section{Conclusions}
\label{sec:discussion}

We developed BCD-SubG for nonconvex optimization over products of
probability simplices and a closed feasible set. The framework combines
feasibility-preserving ECT updates with constrained-subgradient-based
block and pair selection. We established
its global convergence, together with stationarity and rate
estimates under additional eventual-interiority conditions. We derived explicit updates
for the direct-fit and indirect-fit Semi-NNICA models and verified
the convergence conditions. 
Experiments on synthetic mixtures and simulated MRS data showed
generally lower recovery errors and shorter wall-clock time than the
tested Semi-NNICA baselines. In joint diagonalization, at each tested
SNR, one BCD-SubG variant achieved the lowest median final objective
value among the compared methods. Future work will develop more efficient block and pair updates
for larger problems, including parallel strategies inspired
by~\cite{richtarik2016parallel}, while preserving feasibility,
sufficient descent, and convergence guarantees.

\appendix

\crefalias{section}{appendix}

\section{Auxiliary results and proof of the general convergence theorem}
\label{app:block-pair-selection}
\label{app:general-convergence}

\begin{lemma}[Maximal feasible pair]\label{lem:max_ijp_existence}
For the fixed block $p$, set $\vect{\xi}_{\min}^{(p)}
\eqdef
\Proj{\partial_{\Delta_n}f_{(p)}(\vect{u}^{(p)})}(\vect{0})$.
If $\vect{\xi}_{\min}^{(p)} \neq \vect{0}$ and $(i^*,j^*)\in
\operatorname*{arg\,max}_{\substack{1\leq i,j\leq n\\ i\neq j,\ u_j^{(p)}>0}}
\left|
\left\langle
\vect{\xi}_{\min}^{(p)},
\matr{T}'_{i,j}(0)\vect{u}^{(p)}
\right\rangle
\right|$, then the exact distance to the constrained subdifferential at
$\alpha=0$ in \eqref{eq:elementary_subdifferential} coincides with this maximal value:
\begin{equation*}
\operatorname{dist}\!\left(0,\partial h^{(i^*,j^*)}_{(p)}(0)\right) = \left|\left\langle \vect{\xi}_{\min}^{(p)}, \matr{T}'_{i^*,j^*}(0)\vect{u}^{(p)} \right\rangle\right|.
\end{equation*}
\end{lemma}

\begin{proof}
Consider any index pair $(i,j)$ such that $u^{(p)}_i \geq 0$ and $u^{(p)}_j > 0$. By \eqref{eq:ij_transformation}, the feasible interval of $\alpha$ is $[c,1]$, with $c \eqdef -u^{(p)}_i/u^{(p)}_j \leq 0$. Let $\vect{g}^{(p)}
=\nabla f_{(p)}\bigl(\vect{u}^{(p)}\bigr)
=\bigl(g_1^{(p)},\ldots,g_n^{(p)}\bigr)^{\T}
\in\mathbb{R}^{n}$. By the chain rule, the derivative of the elementary function at $\alpha=0$ is
\begin{equation*}
\bigl(h^{(i,j)}_{(p)}\bigr)'(0) = \left\langle \vect{g}^{(p)}, \matr{T}'_{i,j}(0)\vect{u}^{(p)} \right\rangle = \bigl(g^{(p)}_i-g^{(p)}_j\bigr)u^{(p)}_j.
\end{equation*}
The constrained subdifferential of $h^{(i,j)}_{(p)}$ on $[c,1]$ at
$\alpha=0$ is given by $\partial_{[c,1]} h^{(i,j)}_{(p)}(0) = \bigl(h^{(i,j)}_{(p)}\bigr)'(0) + \set{N}_{0}[c,1]$. When $c < 0$, the normal cone $\set{N}_{0}[c,1]$ reduces to $\{0\}$; when $c = 0$, it equals $(-\infty, 0]$. Combining the above cases, the distance from the origin to the subdifferential evaluates to
\begin{equation*}
\operatorname{dist}\!\left(0,\,\partial h^{(i,j)}_{(p)}(0)\right) = 
\begin{cases} 
\bigl|g^{(p)}_i-g^{(p)}_j\bigr|u^{(p)}_j, & \text{if } u^{(p)}_i > 0, \\ 
\bigl(g^{(p)}_j-g^{(p)}_i\bigr)u^{(p)}_j, & \text{if } u^{(p)}_i = 0 \text{ and } \bigl(h^{(i,j)}_{(p)}\bigr)'(0)<0,\\ 
0, & \text{if } u^{(p)}_i = 0 \text{ and } \bigl(h^{(i,j)}_{(p)}\bigr)'(0) \geq 0. 
\end{cases}
\end{equation*}

In view of \eqref{eq:subdifferential_simplex}, the minimum-norm subgradient $\vect{\xi}_{\min}^{(p)} \in \partial_{\Delta_n} f_{(p)}(\vect{u}^{(p)})$ admits the following representation:
\begin{equation*}
\vect{\xi}_{\min}^{(p)} \eqdef \vect{g}^{(p)} + \mu^{(p)} \vect{1} - \sum_{i=1}^n \gamma_i^{(p)} \vect{e}_i=\bigl(\xi_1^{(p)},\ldots,\xi_n^{(p)}\bigr)^\T,
\end{equation*}
where $\mu^{(p)} \in \mathbb{R}$ and $\gamma_i^{(p)} \geq 0$ satisfy $\gamma_i^{(p)} u_i^{(p)} = 0$ for all $i = 1,\dots,n$. Since the chosen index pair $(i,j)$ satisfies $u^{(p)}_i \geq 0$ and $u^{(p)}_j > 0$, the complementarity condition forces $\gamma_j^{(p)}=0$. It then follows that
\begin{equation*}
\left\langle \vect{\xi}_{\min}^{(p)}, \matr{T}'_{i,j}(0)\vect{u}^{(p)} \right\rangle=\left\langle \vect{\xi}_{\min}^{(p)}\bigl(\vect{u}^{(p)}\bigr)^{\T}, \matr{T}'_{i,j}(0) \right\rangle=\bigl(\xi^{(p)}_i-\xi^{(p)}_j\bigr)u^{(p)}_j=\bigl(g^{(p)}_i-g^{(p)}_j-\gamma_i^{(p)}\bigr)u^{(p)}_j.
\end{equation*}

Now, focus on the index pair $(i^*,j^*)$ that makes the absolute value of this inner product maximal. We distinguish two cases with respect to the index $i^*$. Suppose first that $u^{(p)}_{i^*} > 0$. In this scenario, the complementarity condition forces $\gamma_{i^*}^{(p)} = 0$, and we have 
\begin{equation*}
\left|\left\langle \vect{\xi}_{\min}^{(p)}, \matr{T}'_{i^*,j^*}(0)\vect{u}^{(p)} \right\rangle\right|=\bigl|g^{(p)}_{i^*}-g^{(p)}_{j^*}\bigr|u^{(p)}_{j^*} = \operatorname{dist}\!\left(0,\,\partial h^{(i^*,j^*)}_{(p)}(0)\right).
\end{equation*}
Otherwise, we have $u^{(p)}_{i^*}=0$ and hence $\xi^{(p)}_{i^*}\leq0$.
Maximality forces $\xi^{(p)}_{i^*}<0$ and
$\xi^{(p)}_{j^*}>\xi^{(p)}_{i^*}$. Indeed, if
$\xi^{(p)}_{i^*}=0$, pairing the same donor $j^*$ with a suitably signed nonzero component
of $\vect{\xi}_{\min}^{(p)}$ gives a strictly larger value. If
$\xi^{(p)}_{j^*}\leq\xi^{(p)}_{i^*}$, pairing $j^*$ with a positive
component gives the same contradiction. Equation~\eqref{eq:eta_form}
therefore gives $\gamma_{i^*}^{(p)}=0$ and
$\bigl(h^{(i^*,j^*)}_{(p)}\bigr)'(0)<0$. We obtain
\begin{equation*}
\left|\left\langle \vect{\xi}_{\min}^{(p)}, \matr{T}'_{i^*,j^*}(0)\vect{u}^{(p)} \right\rangle\right| = \bigl(\xi^{(p)}_{j^*}-\xi^{(p)}_{i^*}\bigr)u^{(p)}_{j^*} = \bigl(g^{(p)}_{j^*}-g^{(p)}_{i^*}\bigr)u^{(p)}_{j^*} = \operatorname{dist}\!\left(0,\,\partial h^{(i^*,j^*)}_{(p)}(0)\right).
\end{equation*}
Combining both scenarios, we obtain the desired equality and complete the proof.
\end{proof}

\begin{remark}\label{rem:dist_equals_derivative}
The case characterization in the proof of
\cref{lem:max_ijp_existence} shows that, for any index pair $(i,j)$
with $u_j^{(p)}>0$, we have
\begin{equation*}
\operatorname{dist}\!\left(
0,\partial h^{(i,j)}_{(p)}(0)
\right)>0
\quad\Longrightarrow\quad
\operatorname{dist}\!\left(
0,\partial h^{(i,j)}_{(p)}(0)
\right)
=
\left|
\bigl(h^{(i,j)}_{(p)}\bigr)'(0)
\right|.
\end{equation*}
In particular, whenever a $\matr{U}$-update is performed in
\cref{alg:gradient-based}, the selected pair $(i_k,j_k)$ satisfies
$\operatorname{dist}\!\left(
0,\partial h_k(0)
\right)
=
|h_k'(0)|.$ 
Moreover, if the selected pair satisfies
$u_{i_k}^{(p_k)}=0$, then we have
$h_k'(0)<0$.
\end{remark}

We next prove the well-definedness result stated in
\cref{prop:block-pair-existence}.

\begin{proof}[Proof of \cref{prop:block-pair-existence}]

For simplicity, write $\matr{U}=\matr{U}_{k-1}$ and
$\matr{V}=\matr{V}_{k-1}$, and set
$\matr{\Xi}_{\min}
\eqdef\Proj{\partial_{\Delta_n^m}f_{\matr{V}}(\matr{U})}(\matr{0})
=(\vect{\xi}_{\min}^{(1)},\ldots,\vect{\xi}_{\min}^{(m)})$.
If $\partial_{\Delta_{n}^{m}} f_{\matr{V}}(\matr{U})$ contains $\matr{0}$, choose any block $p$, any $j$ with $u_j^{(p)}>0$, and any $i\neq j$; then \eqref{eq:grad_inequality_h} holds trivially. Hence, we assume that
$\matr{0}\notin\partial_{\Delta_n^m}f_{\matr{V}}(\matr{U})$,
so that $\matr{\Xi}_{\min}\neq\matr{0}$. 
For the minimum-norm subgradient $\matr{\Xi}_{\min}$, there exists a block $p \in \{1, \dots, m\}$ such that
\begin{equation*}
\left\|\vect{\xi}_{\min}^{(p)}\right\| \geq \frac{1}{\sqrt{m}}\left\|\matr{\Xi}_{\min}\right\|.
\end{equation*}
We choose this index as the updated block and set $p_k=p$. 
Let $\matr{M} \eqdef \vect{\xi}_{\min}^{(p)}\bigl(\vect{u}^{(p)}\bigr)^{\T}$, where $M_{i,j} = \xi^{(p)}_i u^{(p)}_j$ denotes its entries. Since $-\vect{\xi}_{\min}^{(p)}$ lies in the tangent cone $\set{T}_{\vect{u}^{(p)}}\Delta_{n}$, it satisfies $\sum_{i=1}^n \xi^{(p)}_i = 0$. Consequently, for each column $j$, we have $\sum_{i=1}^n M_{i,j} = 0$. To simplify the subsequent notation, we define the inner product evaluation as $p_{i,j}$, which reduces to 
\begin{equation*}
p_{i,j} \eqdef \left\langle \vect{\xi}_{\min}^{(p)}, \matr{T}'_{i,j}(0)\vect{u}^{(p)} \right\rangle = M_{i,j} - M_{j,j}.
\end{equation*}
For each fixed $j$, using $\sum_{i=1}^{n}M_{i,j}=0$, we obtain
\begin{equation*}
\sum_{i\neq j}|p_{i,j}|^2=
\sum_{i=1}^{n}|M_{i,j}-M_{j,j}|^2=\sum_{i=1}^{n}M_{i,j}^2
-2M_{j,j}\sum_{i=1}^{n}M_{i,j}
+nM_{j,j}^2\geq
\sum_{i=1}^{n}M_{i,j}^2.
\end{equation*}
Summing the above inequality over $j=1,\ldots,n$ and using the fact
that there are $n(n-1)$ ordered pairs $(i,j)$ with $i\neq j$, we obtain
\begin{equation*}
\max_{\substack{1\leq i,j\leq n\\i\neq j}}|p_{i,j}|
\geq
\frac{1}{\sqrt{n(n-1)}}
\left(
\sum_{j=1}^{n}\sum_{i\neq j}|p_{i,j}|^2
\right)^{1/2}
\geq
\frac{1}{\sqrt{n(n-1)}}\|\matr{M}\|.
\end{equation*}
Since $\vect{u}^{(p)} \in \Delta_n$, its Euclidean norm satisfies $\left\|\vect{u}^{(p)}\right\| \geq \frac{1}{\sqrt{n}}$. It follows that
\begin{equation*}
\left\|\matr{M}\right\| = \left\|\vect{\xi}_{\min}^{(p)}\right\|\left\|\vect{u}^{(p)}\right\| \geq \frac{1}{\sqrt{n}}\left\|\vect{\xi}_{\min}^{(p)}\right\| \geq \frac{1}{\sqrt{mn}}\left\|\matr{\Xi}_{\min}\right\|=\frac{1}{\sqrt{mn}}\operatorname{dist}\!\left(\matr{0}, \partial_{\Delta_{n}^{m}} f_{\matr{V}}(\matr{U})\right).
\end{equation*}
Pairs with $u_j^{(p)}=0$ have $p_{i,j}=0$, whereas the maximum above is
positive. We may therefore choose $(i,j)\in
\operatorname*{arg\,max}_{\substack{1\leq i,j\leq n\\i\neq j,\ u_j^{(p)}>0}}
|p_{i,j}|$ and set $(i_k,j_k)=(i,j)$. By \cref{lem:max_ijp_existence}, we have
\begin{equation*}
\operatorname{dist}\!\left(0,\partial h^{(i,j)}_{(p)}(0)\right)=\left|\left\langle \vect{\xi}_{\min}^{(p)}, \matr{T}'_{i,j}(0)\vect{u}^{(p)} \right\rangle\right|=\max_{\substack{1\leq i,j\leq n\\i\neq j}} |p_{i,j}|.
\end{equation*}

Combining the above inequalities, there always exists an index pair $(i_k,j_k)$ and a block $p_k$ such that
\begin{equation*}
\operatorname{dist}\!\left(0,\partial h^{(i_k,j_k)}_{(p_k)}(0)\right)\geq \frac{1}{n\sqrt{m(n-1)}}\operatorname{dist}\!\left(\matr{0}, \partial_{\Delta_{n}^{m}} f_{\matr{V}}(\matr{U})\right). \qedhere
\end{equation*}
\end{proof}

\begin{remark}\label{rem:pair-reversal}
Without loss of generality, the index pair $(i,j)$ selected in
\cref{alg:gradient-based} can be considered to satisfy $u_i^{(p)}\leq u_j^{(p)}$. 
Indeed, if the selected index pair satisfies
$u_i^{(p)}>u_j^{(p)}>0$, then swap $i$ and $j$, and set $\widetilde{\alpha}=-(u_j^{(p)}/u_i^{(p)})\alpha$. We immediately obtain $\matr{T}_{i,j}(\alpha)\vect{u}^{(p)}
=
\matr{T}_{j,i}(\widetilde{\alpha})\vect{u}^{(p)}$, and the feasible interval for $\alpha$ is mapped bijectively onto the
feasible interval for $\widetilde{\alpha}$.

Since $u_i^{(p)},u_j^{(p)}>0$, the origin is an interior point of both
feasible intervals. Hence, we have
\[
\operatorname{dist}\!\left(0,\partial h^{(j,i)}_{(p)}(0)\right)
=
\frac{u_i^{(p)}}{u_j^{(p)}}
\operatorname{dist}\!\left(
0,\partial h^{(i,j)}_{(p)}(0)
\right)
\geq
\operatorname{dist}\!\left(
0,\partial h^{(i,j)}_{(p)}(0)
\right).
\]
Therefore, the reversed pair $(j,i)$ also satisfies
\eqref{eq:grad_inequality_h}.
If $\alpha^*$ is an output of either line search rule for $(i,j)$, then
$\widetilde{\alpha}^*=-(u_j^{(p)}/u_i^{(p)})\alpha^*$ is a corresponding output of the
same rule for $(j,i)$ and produces the same ECT update. Whenever a rule has
more than one eligible minimizer, we use this correspondence as the
tie-breaking convention. This justifies assuming
$u_i^{(p)}\leq u_j^{(p)}$.
\end{remark}

We next establish the constrained \L{}ojasiewicz inequality on the product space
$\Delta_n^m\times\Omega$ used in the convergence analysis.

\begin{lemma}[Product space constrained \L{}ojasiewicz inequality]\label{lem:product-lojasiewicz}
Let $f$ admit a real-analytic extension to a neighborhood of
$\Delta_n^m\times\Omega$. Assume that $\Omega$ is a closed feasible set and that, at every $\matr{V}^*\in\Omega$, it admits a
real-analytic parametrization satisfying conditions~(i)--(ii) of
\cref{lem:nonsmooth-lojasiewicz}. Then, at each
$(\matr{U}^*,\matr{V}^*)\in\Delta_n^m\times\Omega$, there exist
$r_{\mathrm{L}},C_{\mathrm{L}}>0$ and $\theta\in(0,1/2]$ such that every feasible
$(\matr{U},\matr{V})$ satisfying
$\|(\matr{U},\matr{V})-(\matr{U}^*,\matr{V}^*)\|<r_{\mathrm{L}}$ obeys
\begin{equation*}\label{eq:lojasiewicz_two}
\left|f(\matr{U},\matr{V})-f(\matr{U}^*,\matr{V}^*)\right|^{1-\theta}\leq C_{\mathrm{L}}\left(\operatorname{dist}\!\left(\matr{0},\partial_{\Delta_n^m}f_{\matr{V}}(\matr{U})\right)^2+g^{-}_{f_{\matr{U}}}(\matr{V})^2\right)^{1/2}. \tag{\textbf{\L{}-UV}}
\end{equation*}
\end{lemma}

\begin{proof}
Fix $(\matr{U}^*,\matr{V}^*)$ and set
$\set{U}_{\matr{U}}\eqdef
(\mathbb{R}^n\setminus\{\vect{0}\})^m$. Define
\[
\tau_{\matr{U}}(\matr{W})_{i,p}
\eqdef\frac{W_{i,p}^2}{\sum_{r=1}^nW_{r,p}^2},
\ \  
\sigma(\matr{U})_{i,p}\eqdef\sqrt{U_{i,p}}.
\]
The map $\tau_{\matr{U}}$ is real analytic on its open domain, while
$\sigma$ is continuous and
$\tau_{\matr{U}}\circ\sigma$ is the identity on $\Delta_n^m$.
Thus, if $\matr{W}^*=\sigma(\matr{U}^*)$ and
$\set{W}_{\matr{U}}$ is any open neighborhood of $\matr{W}^*$, then
$\sigma^{-1}(\set{W}_{\matr{U}})$ is a relatively open
neighborhood of $\matr{U}^*$ contained in
$\tau_{\matr{U}}(\set{W}_{\matr{U}})$. Hence the simplex map
satisfies conditions~(i)--(ii), including at boundary points.

Let $\tau_{\matr{V}}:\set{U}_{\matr{V}}\to\Omega$ be the local
parametrization at $\matr{V}^*$. The product map
$\tau(\matr{W},\vect{z})
\eqdef(\tau_{\matr{U}}(\matr{W}),\tau_{\matr{V}}(\vect{z}))$
also satisfies conditions~(i)--(ii): every neighborhood of the preimage
contains a product of neighborhoods, whose image contains a relatively open
product neighborhood of $(\matr{U}^*,\matr{V}^*)$.
Because $\Delta_n^m\times\Omega$ is closed,
\cref{lem:nonsmooth-lojasiewicz} applies. 
Finally, the tangent cone of the product is
$\set{T}_{\matr{U}}\Delta_n^m\times \set{T}_{\matr{V}}\Omega$.
The product geometry therefore gives
\begin{equation}\label{eq:dist_grad_g}
g^{-}_f(\matr{U},\matr{V})^2
=
\operatorname{dist}\!\left(
\matr{0},
\partial_{\Delta_n^m}f_{\matr{V}}(\matr{U})
\right)^2
+
g^{-}_{f_{\matr{U}}}(\matr{V})^2.
\end{equation}
Combining this identity with \eqref{eq:lojasiewicz} proves the result.
\end{proof}

We now prove \cref{thm:Schneider_frame_two} using
\cref{lem:product-lojasiewicz} and the abstract convergence theorem.

\begin{proof}[Proof of \cref{thm:Schneider_frame_two}]
Let $\vect{x}_k=\vecl(\matr{U}_k,\matr{V}_k)$, using the
isometric identification specified in the notation section and the same
symbol $f$ for the objective in these coordinates. Then, we have $\|\vect{x}_k-\vect{x}_{k-1}\|=\sqrt{\|\matr{U}_k-\matr{U}_{k-1}\|^2+\|\matr{V}_k-\matr{V}_{k-1}\|^2}$. By the ECT update rule \eqref{eq:ij_transformation}, when the $p_k$-column of $\matr{U}$ is updated along the index pair
$(i_k,j_k)$, let $b_k \eqdef (\matr{U}_{k-1})_{j_k,p_k}$, we have $\|\matr{U}_k-\matr{U}_{k-1}\|
=
\sqrt{2}b_k|\alpha_k^*|$. 

\subsubsection*{Verification of \eqref{cond:A1}} 
Suppose first that \eqref{eq:update_selection_inequality} holds. 
It then follows immediately from the characterization of $g^{-}_f(\vect{x})$ in \eqref{eq:dist_grad_g} that
$g^{-}_{f_{\matr{U}_{k-1}}}(\matr{V}_{k-1})\geq \sqrt{2}g^{-}_f(\vect{x}_{k-1})/2$.
Since $\matr{U}_k=\matr{U}_{k-1}$, we obtain from \eqref{cond:V1} that
\begin{equation*}
f(\vect{x}_{k-1})-f(\vect{x}_k)=f_{\matr{U}_{k-1}}(\matr{V}_{k-1})
-
f_{\matr{U}_{k-1}}(\matr{V}_{k})
\geq
\frac{\sqrt{2}\eta_{\mathrm{V}}}{2}
g^{-}_f(\vect{x}_{k-1})
\|\vect{x}_k-\vect{x}_{k-1}\|.
\end{equation*}
Otherwise, we obtain $\operatorname{dist}\!\left(\matr{0},\partial_{\Delta_n^m}f_{\matr{V}_{k-1}}(\matr{U}_{k-1})\right)\geq \sqrt{2}g^{-}_f(\vect{x}_{k-1})/2$. By the definition of the elementary function $h_k(\alpha)$ in
\eqref{eq:elementary_function} and \eqref{eq:grad_inequality_h}, we have $f_{\matr{V}_{k-1}}(\matr{U}_{k-1})=h_k(0)$ and $f_{\matr{V}_{k-1}}(\matr{U}_k)=h_k(\alpha_k^*)$. Since $\matr{V}_k=\matr{V}_{k-1}$ and $0<b_k\leq1$, we have $\|\vect{x}_k-\vect{x}_{k-1}\|
=
\sqrt{2}\,b_k|\alpha_k^*|
\leq
\sqrt{2}|\alpha_k^*|$. Combining this bound with \eqref{cond:U1} and \eqref{eq:grad_inequality_h}, we obtain
\begin{equation*}
f(\vect{x}_{k-1})-f(\vect{x}_k)
=h_k(0)
-
h_k(\alpha_k^*)\geq
\frac{\eta_{\mathrm{U}}\delta}{2}
g^{-}_f(\vect{x}_{k-1})
\|\vect{x}_k-\vect{x}_{k-1}\|.
\end{equation*}
Therefore, \eqref{cond:A1} holds with $\eta
=
\min\left\{
\frac{\sqrt{2}\eta_{\mathrm{V}}}{2},
\frac{\eta_{\mathrm{U}}\delta}{2}
\right\}$.

\subsubsection*{Verification of \eqref{cond:A2}}

If $g^{-}_f(\vect{x}_{k-1})=0$, then \eqref{eq:dist_grad_g} gives
\[
\operatorname{dist}\!\left(\matr{0},
\partial_{\Delta_n^m}f_{\matr{V}_{k-1}}(\matr{U}_{k-1})\right)=0,
\\ \ 
g^{-}_{f_{\matr{U}_{k-1}}}(\matr{V}_{k-1})=0.
\]
Hence \eqref{eq:update_selection_inequality} holds, and
\eqref{cond:V2} gives $\matr{V}_k=\matr{V}_{k-1}$, while
$\matr{U}_k=\matr{U}_{k-1}$ by construction.
Thus \eqref{cond:A2} holds. 
By \cref{lem:product-lojasiewicz}, \eqref{eq:lojasiewicz_two} holds.
Therefore, \cref{thm:Schneider_frame} implies that the whole sequence converges to
$(\matr{U}^*,\matr{V}^*)$.

\subsubsection*{Verification of \eqref{cond:A3}}

If \eqref{eq:update_selection_inequality} holds, then by \eqref{cond:V3}, it follows that 
\begin{equation*}
\|\vect{x}_k-\vect{x}_{k-1}\|=\|\matr{V}_k-\matr{V}_{k-1}\|
\geq
\frac{\sqrt{2}\kappa_{\mathrm{V}}}{2}
g^{-}_f(\vect{x}_{k-1}).
\end{equation*}
Otherwise, applying \cref{assump:interiority} to the $\matr{U}$-block, we have $b_k\geq\varepsilon_{\mathrm{int}}$ for all sufficiently
large $k$. Using \eqref{cond:U3} and \eqref{eq:grad_inequality_h}, we obtain
\begin{equation*}
\|\vect{x}_k-\vect{x}_{k-1}\|=\sqrt{2}\,b_k|\alpha_k^*|\geq \varepsilon_{\mathrm{int}}\delta\kappa_{\mathrm{U}}g^{-}_f(\vect{x}_{k-1}).
\end{equation*}
Hence \eqref{cond:A3} holds with $\kappa
=
\min\left\{
\frac{\sqrt{2}\kappa_{\mathrm{V}}}{2},
\varepsilon_{\mathrm{int}}\delta\kappa_{\mathrm{U}}
\right\}$. 
By \cref{thm:Schneider_frame}, we have
$g^{-}_f(\vect{x}_k)\to0$, and the convergence rate in
\eqref{eq:conver_rate} follows. Moreover, by
\eqref{eq:dist_grad_g}, we have 
\[
\operatorname{dist}\!\left(
\matr{0},
\partial_{\Delta_n^m}f_{\matr{V}_k}(\matr{U}_k)
\right)\to0,
\\ \ 
g^{-}_{f_{\matr{U}_{k}}}(\matr{V}_{k})\to0.
\]
Let $\matr{\Xi}_{k,\min}$ be the minimum-norm element of
$\partial_{\Delta_n^m}f_{\matr{V}_k}(\matr{U}_k)$.
Then we have $\matr{\Xi}_{k,\min}\to\matr{0}$. Since $\nabla_{\matr{U}}f$ is continuous
and the normal-cone mapping of the closed convex set $\Delta_n^m$ has a
closed graph, passing to the limit gives $\matr{0}\in\partial_{\Delta_n^m}
f_{\matr{V}^{*}}(\matr{U}^{*})$. By the lower semicontinuity of
$(\matr{U},\matr{V})\mapsto g^-_{f_{\matr{U}}}(\matr{V})$, we have 
\[
g^-_{f_{\matr{U}^*}}(\matr{V}^*)
\leq
\liminf_{k\to\infty}
g^-_{f_{\matr{U}_k}}(\matr{V}_k)
=0.
\]
Hence, we have $g^-_{f_{\matr{U}^*}}(\matr{V}^*)=0$, and
$(\matr{U}^*,\matr{V}^*)$ is stationary.
\end{proof}

\section{Proof of the direct-fit update lemmas}
\label{app:proof-direct-a}

\begin{proof}[Proof of \cref{lem:A-update-direct}]
Recall that $h_k$ is the elementary function defined in
\eqref{eq:hk_direct}. If $\operatorname{dist}(0,\partial h_k(0))=0$, then
$\alpha_k^*=0$ by \hyperlink{rule:R2}{Rule~(R2)}, and all three
statements hold trivially. Hence, in the remainder of the proof, we assume
$\operatorname{dist}(0,\partial h_k(0))>0.$

We first prove part~(i). Fix an update iteration $k$ and set
$p=p_k$, $i=i_k$, and $j=j_k$. For a trial step $\alpha$, let $\matr{A}(\alpha)$
be obtained from $\matr{A}_{k-1}$ by replacing its
$p$-th column with
$\matr{T}_{i,j}(\alpha)(\matr{A}_{k-1})_{:,p}$. The ECT transformation changes
only two entries:
\[
(\matr{A}(\alpha))_{i,p}
=(\matr{A}_{k-1})_{i,p}+\alpha(\matr{A}_{k-1})_{j,p},
\\ \ 
(\matr{A}(\alpha))_{j,p}
=(1-\alpha)(\matr{A}_{k-1})_{j,p}.
\]
Set $a \eqdef (\matr{A}_{k-1})_{i,p}\geq0$ and
$b \eqdef (\matr{A}_{k-1})_{j,p}>0$, so the feasible interval is
$[-a/b,1]$. By \cref{rem:pair-reversal}, we may assume that $a\leq b$.
Define
$\matr{R}^{(\ell)}_{k-1}\eqdef\matr{C}^{(\ell)}-
\matr{A}_{k-1}\matr{D}^{(\ell)}_{k-1}\matr{A}^{\T}_{k-1}$.
A direct
expansion under the ECT update gives the quartic polynomial
\begin{equation}\label{eq:h_quartic}
h_k(\alpha) - h_k(0) = c_4\alpha^4 + c_3\alpha^3 + c_2\alpha^2 + c_1\alpha,
\end{equation}
with the coefficients given by
\begin{align*}
c_4 &= 4b^4\sum_{\ell=1}^{L} \bigl(D_{p,p}^{(\ell)}\bigr)^2,\ \ c_3 = 8(a-b)b^3\sum_{\ell=1}^{L}
\bigl(D_{p,p}^{(\ell)}\bigr)^2, \\
c_2 &= \left(4b^2\sum_{q=1}^{n} A_{q,p}^2
+2b^2(b-a)^2\right)
\sum_{\ell=1}^{L}\bigl(D_{p,p}^{(\ell)}\bigr)^2 
-2b^2\sum_{\ell=1}^{L}D_{p,p}^{(\ell)}
\left(R_{i,i}^{(\ell)}+R_{j,j}^{(\ell)}-2R_{i,j}^{(\ell)}\right), \\
c_1 &= 4b\sum_{\ell=1}^{L}D_{p,p}^{(\ell)}
\sum_{q=1}^{n}A_{q,p}
\left(R_{j,q}^{(\ell)}-R_{i,q}^{(\ell)}\right),
\end{align*}
where $A_{q,p}$, $D_{p,p}^{(\ell)}$, and $R_{i,j}^{(\ell)}$ denote the
corresponding entries of $\matr{A}_{k-1}$,
$\matr{D}^{(\ell)}_{k-1}$, and $\matr{R}^{(\ell)}_{k-1}$, respectively.

By \eqref{eq:h_quartic} and \cref{rem:dist_equals_derivative}, we have 
$h'_k(0)=c_1$ and 
$\operatorname{dist}\!\left(0,\partial h_k(0)\right)
=|h'_k(0)|=|c_1|.$ 
If $c_4=0$, the coefficient formulas give $c_3=c_2=c_1=0$, which
contradicts the standing assumption. Thus, we have $c_4>0$. Let
$\chi\eqdef(b-a)/b$. Then we have $c_3=-2\chi c_4$ and, because $a\leq b$, we obtain
$-1\leq-a/b\leq0\leq\chi\leq1$.

Let $\alpha_k^*\in[-a/b,1]$ denote the first local minimizer selected by
the line search. If $h_k'(\alpha_k^*)=0$, then we have
$\alpha_k^*h_k'(\alpha_k^*)=0$.
Otherwise, $\alpha_k^*$ is a boundary point, and no interior local minimizer
is encountered when moving from the origin along the descent direction.
Hence, the derivative retains the descent sign up to $\alpha_k^*$, so
\[
\alpha_k^*h_k'(\alpha_k^*)
=4c_4(\alpha_k^*)^4+3c_3(\alpha_k^*)^3
+2c_2(\alpha_k^*)^2+c_1\alpha_k^*\leq0.
\]
Moreover, \hyperlink{rule:R2}{Rule~(R2)} follows a descent direction that
$c_1\alpha_k^*=-|c_1||\alpha_k^*|$.
Combining this inequality with \eqref{eq:h_quartic} yields
\begin{align*}
h_k(0)-h_k(\alpha_k^*)
&=-c_4(\alpha_k^*)^4-c_3(\alpha_k^*)^3
-c_2(\alpha_k^*)^2-c_1\alpha_k^*
\geq c_4(\alpha_k^*)^4
+\frac{1}{2}c_3(\alpha_k^*)^3
-\frac{1}{2}c_1\alpha_k^*\\
&=\frac{1}{2}|c_1||\alpha_k^*|
+c_4(\alpha_k^*)^3(\alpha_k^*-\chi).
\end{align*}
If $\alpha_k^*\in[-a/b,0]\cup[\chi,1]$, this gives
$h_k(0)-h_k(\alpha_k^*)
\geq\frac{1}{2}|c_1||\alpha_k^*|.$ 
It remains to consider $\alpha_k^*\in(0,\chi)$. The line search then
proceeds in the positive descent direction, so $c_1=h_k'(0)<0$. Hence, we obtain
\begin{align*}
h_k(0)-h_k(\alpha_k^*)
\geq \frac{1}{2}|c_1||\alpha_k^*|
+c_4(\alpha_k^*)^3(\alpha_k^*-\chi)
=\left(\frac{1}{2}
+\frac{c_4(\alpha_k^*)^2(\alpha_k^*-\chi)}{|c_1|}\right)
|c_1||\alpha_k^*|.
\end{align*}
We next prove that the term $\frac{1}{2}
+\frac{c_4(\alpha_k^*)^2(\alpha_k^*-\chi)}{|c_1|}$ admits a uniform positive lower bound.

First, suppose that $h_k'$ has three distinct real roots
$r_1<r_2<r_3$. Vieta's theorem gives
$r_1+r_2+r_3=3\chi/2>0$. Since $c_1=h_k'(0)<0$, one negative root and
two positive roots are impossible. If two roots were negative, the positive
root would exceed $3\chi/2>\chi$, contradicting
$\alpha_k^*\in(0,\chi)$. Thus all three roots are positive, and we have
$0<\alpha_k^*=r_1<\chi/2$. Because $\alpha_k^*$ is an interior local
minimizer, we have
\begin{align*}
h_k''(\alpha_k^*)
=12c_4(\alpha_k^*)^2-12c_4\chi\alpha_k^*+2c_2\geq0,\ \
h_k'(\alpha_k^*)
=4c_4(\alpha_k^*)^3-6c_4\chi(\alpha_k^*)^2
+2c_2\alpha_k^*-|c_1|=0.
\end{align*}
These relations yield
$|c_1|\geq6c_4\chi(\alpha_k^*)^2-8c_4(\alpha_k^*)^3>0.$ 
Since $c_4(\alpha_k^*)^2(\alpha_k^*-\chi)<0$, we obtain
\[
\frac{1}{2}
+\frac{c_4(\alpha_k^*)^2(\alpha_k^*-\chi)}{|c_1|}
\geq
\frac{2\chi-3\alpha_k^*}{6\chi-8\alpha_k^*}
\geq\frac{1}{4},
\]
where the right-hand side decreases to $\frac{1}{4}$ as
$\alpha_k^*\to\chi/2$ from below.

Second, suppose that $h_k'$ does not have three distinct real roots.
Since $\alpha_k^*\in(0,\chi)$ is an interior local minimizer, it cannot
be a root of even multiplicity. Counting multiplicity, write the other two
roots as $r\pm\mathrm{i}s$, where $r,s\in\mathbb{R}$ and $s\geq0$.
Vieta's relations give
$\alpha_k^*(r^2+s^2)=|c_1|/(4c_4)$ and
$\alpha_k^*+2r=3\chi/2$. Thus, we have $r>0$ and obtain
\[
|c_1|
\geq4c_4\alpha_k^*r^2
=c_4\alpha_k^*
\left(\frac{3}{2}\chi-\alpha_k^*\right)^2
>0.
\]
Consequently, we obtain
\[
\frac{1}{2}
+\frac{c_4(\alpha_k^*)^2(\alpha_k^*-\chi)}{|c_1|}
\geq
\frac{1}{2}
+\frac{\alpha_k^*(\alpha_k^*-\chi)}
{(3\chi/2-\alpha_k^*)^2}
\geq\frac{1}{6},
\]
where the last rational expression is minimized at
$\alpha_k^*=3\chi/4$. 
Combining the cases gives
\[
h_k(0)-h_k(\alpha_k^*)
\geq\frac{1}{6}|c_1||\alpha_k^*|
=\eta_{\mathrm{A},1}\operatorname{dist}\!\left(0,\partial h_k(0)\right)
|\alpha_k^*|,
\]
which proves part~(i).

Part~(ii) follows directly from \hyperlink{rule:R2}{Rule~(R2)}.

Finally, we prove part~(iii).
Suppose that \cref{assump:interiority} holds for the $\matr{A}$-block.
For all sufficiently large $k$, $\alpha_k^*$ cannot attain either endpoint
of the feasible interval, since $\alpha_k^*=1$ would give
$(\matr{A}_k)_{j,p}=0$, while $\alpha_k^*=-a/b$ would give
$(\matr{A}_k)_{i,p}=0$. Hence, $\alpha_k^*$ is an interior local minimizer
and $h_k'(\alpha_k^*)=0$, which gives
$4c_4(\alpha_k^*)^3
+3c_3(\alpha_k^*)^2
+2c_2\alpha_k^*
+c_1=0.$ 
Define
$\tilde{c}_4\eqdef c_4/b^2$,
$\tilde{c}_3\eqdef c_3/b^2$, and
$\tilde{c}_2\eqdef c_2/b^2$. Then we have
\[
|\alpha_k^*|
=\frac{|c_1|}{\left|4c_4(\alpha_k^*)^2+3c_3\alpha_k^*+2c_2\right|}
=\frac{|c_1|}{b^2\left|4\tilde{c}_4(\alpha_k^*)^2
+3\tilde{c}_3\alpha_k^*+2\tilde{c}_2\right|}.
\]
By \cref{prop:boundedness-direct}, the iterates
$\matr{D}_{k-1}^{(\ell)}$ are uniformly bounded for fixed $\lambda>0$. Moreover, we have
$\matr{A}_{k-1}\in\Delta_n^m$ and, because $a\leq b$, we obtain
$|\alpha_k^*|\leq1$. The coefficient formulas therefore show that
$\tilde{c}_2$, $\tilde{c}_3$, and $\tilde{c}_4$ are uniformly bounded.
Hence, there is a constant $K_{\mathrm{A},1}>0$, independent of $k$,
that bounds
$|4\tilde{c}_4(\alpha_k^*)^2+3\tilde{c}_3\alpha_k^*+2\tilde{c}_2|$
from above. With $\kappa_{\mathrm{A},1}=1/K_{\mathrm{A},1}$, we obtain
\[
|\alpha_k^*|
\geq
\frac{|c_1|}{b^2K_{\mathrm{A},1}}
\geq
\kappa_{\mathrm{A},1}\operatorname{dist}\!\left(0,\partial h_k(0)\right),
\]
where the last inequality uses $0<b\leq1$. This proves part~(iii).
\end{proof}

\begin{proof}[Proof of \cref{lem:D-update-direct}]
We first prove part~(i). For each $\ell$, let
$\vect{d}^{(\ell)}\eqdef\operatorname{diag}(\matr{D}^{(\ell)})$,
$\vect{c}^{(\ell)}\eqdef\vecl(\matr{C}^{(\ell)})$, and
$\matr{B}_{k-1}\eqdef\matr{A}_{k-1}\kr\matr{A}_{k-1}$. Vectorization
and
$\vecl(\matr{A}\operatorname{Diag}(\vect{d})\matr{A}^{\T})
=(\matr{A}\kr\matr{A})\vect{d}$ give
\[
f_{\matr{A}_{k-1}}
\bigl((\operatorname{Diag}(\vect{d}^{(\ell)}))_{\ell=1}^{L}\bigr)
=\sum_{\ell=1}^{L}
\left(\left\|\vect{c}^{(\ell)}-\matr{B}_{k-1}\vect{d}^{(\ell)}\right\|^2
+\lambda\|\vect{d}^{(\ell)}\|^2\right).
\]
For each $\ell$, the normal equations are
$(\matr{B}_{k-1}^{\T}\matr{B}_{k-1}+\lambda\matr{I}_m)
\vect{d}_k^{(\ell)}=\matr{B}_{k-1}^{\T}\vect{c}^{(\ell)}$.
Since $\lambda>0$, their unique solution gives \eqref{eq:D-update-1}.
This is the ridge-regularized counterpart of the least-squares update
in~\cite{yeredor2002non}. 
Stack the component vectors as
\[
\vect{d}\eqdef\bigl((\vect{d}^{(1)})^{\T},\ldots,
(\vect{d}^{(L)})^{\T}\bigr)^{\T},\\ \ 
\vect{c}\eqdef\bigl((\vect{c}^{(1)})^{\T},\ldots,
(\vect{c}^{(L)})^{\T}\bigr)^{\T}.
\]
Using this stacked notation, define
\[
f_{\mathrm{D},k}(\vect{d})
\eqdef f_{\matr{A}_{k-1}}
\bigl((\operatorname{Diag}(\vect{d}^{(\ell)}))_{\ell=1}^{L}\bigr)
=\left\|\vect{c}-(\matr{I}_L\kron\matr{B}_{k-1})\vect{d}\right\|^2
+\lambda\|\vect{d}\|^2.
\]
With
$\matr{H}_{k-1}\eqdef\matr{B}_{k-1}^{\T}\matr{B}_{k-1}
+\lambda\matr{I}_m
=\bigl(\matr{A}_{k-1}^{\T}\matr{A}_{k-1}\bigr)
\hadam
\bigl(\matr{A}_{k-1}^{\T}\matr{A}_{k-1}\bigr)+\lambda\matr{I}_m,$ 
its gradient and Hessian are
\begin{align*}
\nabla f_{\mathrm{D},k}(\vect{d})
=2(\matr{I}_L\kron\matr{H}_{k-1})\vect{d}
-2(\matr{I}_L\kron\matr{B}_{k-1})^{\T}\vect{c},\ \ 
\nabla^2 f_{\mathrm{D},k}(\vect{d})
=2(\matr{I}_L\kron\matr{H}_{k-1}).
\end{align*}
It follows that
$\nabla f_{\mathrm{D},k}(\vect{d}_k)
=\nabla f_{\mathrm{D},k}(\vect{d}_{k-1})
+2(\matr{I}_L\kron\matr{H}_{k-1})
(\vect{d}_k-\vect{d}_{k-1}).$ 
Combining this identity with the first-order condition
$\nabla f_{\mathrm{D},k}(\vect{d}_k)=\vect{0}$ gives
\begin{equation}\label{eq:D-gradient-identity}
\nabla f_{\mathrm{D},k}(\vect{d}_{k-1})
=-2(\matr{I}_L\kron\matr{H}_{k-1})
(\vect{d}_k-\vect{d}_{k-1}).
\end{equation}
Let $K_{\mathrm{D},k}\eqdef\lambda_{\max}(2\matr{H}_{k-1})$.
Writing the columns of $\matr{A}_{k-1}$ as $\vect{a}_p$, simplex
feasibility gives $\|\vect{a}_p\|\leq1$. Thus, we have
$\|\matr{B}_{k-1}\|_2^2\leq\|\matr{B}_{k-1}\|_{\mathrm F}^2
=\sum_{p=1}^{m}\|\vect{a}_p\|^4\leq m$.
Consequently, we have $K_{\mathrm{D},k}\leq K_{\mathrm{D}}\eqdef2(m+\lambda)$ for every $k$, and we obtain
\begin{equation}\label{eq:D-gradient-maxbound}
\left\|\nabla f_{\mathrm{D},k}(\vect{d}_{k-1})\right\|
\leq
K_{\mathrm{D},k}
\left\|\vect{d}_k-\vect{d}_{k-1}\right\|
\leq
K_{\mathrm{D}}
\left\|\vect{d}_k-\vect{d}_{k-1}\right\|.
\end{equation}
For every $k$, the fixed quadratic penalty gives
$\matr{H}_{k-1}
=\matr{B}_{k-1}^{\T}\matr{B}_{k-1}
+\lambda\matr{I}_m
\succeq\lambda\matr{I}_m$.
Consequently, we have
\begin{equation}\label{eq:D-gradient-minbound}
\matr{I}_L\kron\matr{H}_{k-1}
\succeq
\lambda\matr{I}_{Lm}.
\end{equation}
By the quadratic expansion of $f_{\mathrm{D},k}$, we have
\begin{align*}
f_{\mathrm{D},k}(\vect{d}_k)
={}f_{\mathrm{D},k}(\vect{d}_{k-1})
+\left\langle\nabla f_{\mathrm{D},k}(\vect{d}_{k-1}),
\vect{d}_k-\vect{d}_{k-1}\right\rangle
+(\vect{d}_k-\vect{d}_{k-1})^{\T}
(\matr{I}_L\kron\matr{H}_{k-1})
(\vect{d}_k-\vect{d}_{k-1}).
\end{align*}
Substituting \eqref{eq:D-gradient-identity} into this expansion and then
applying \eqref{eq:D-gradient-maxbound} and
\eqref{eq:D-gradient-minbound} gives
\begin{align*}
f_{\matr{A}_{k-1}}(\matr{D}_{k-1})
-f_{\matr{A}_{k-1}}(\matr{D}_k)
&=f_{\mathrm{D},k}(\vect{d}_{k-1})-f_{\mathrm{D},k}(\vect{d}_k)
=(\vect{d}_k-\vect{d}_{k-1})^{\T}
(\matr{I}_L\kron\matr{H}_{k-1})
(\vect{d}_k-\vect{d}_{k-1})\\
&\geq\lambda\|\vect{d}_k-\vect{d}_{k-1}\|^2
\geq \frac{\lambda}{K_{\mathrm{D}}}
\left\|\nabla f_{\mathrm{D},k}(\vect{d}_{k-1})\right\|
\|\matr{D}_{k}-\matr{D}_{k-1}\|.
\end{align*}
Because $\operatorname{Diag}:\mathbb{R}^m\to\set{D}_m$ is a linear
isometry under the Euclidean and Frobenius norms, the chain rule and the
stacked representation give
$\left\|\nabla f_{\mathrm{D},k}(\vect{d}_{k-1})\right\|
=
\left\|
\nabla f_{\matr{A}_{k-1}}\!\left(\matr{D}_{k-1}\right)
\right\|.$
Since $\eta_{\mathrm{D}}=\lambda/K_{\mathrm{D}}$, this proves part~(i).

We next prove part~(ii). If $\left\|
\nabla f_{\matr{A}_{k-1}}\!\left(\matr{D}_{k-1}\right)
\right\|$ vanishes, \eqref{eq:D-gradient-identity} and
the positive definiteness of $\matr{I}_L\kron\matr{H}_{k-1}$ give
$\vect{d}_k=\vect{d}_{k-1}$, as required.

Finally, \eqref{eq:D-gradient-maxbound} gives
\[
\|\matr{D}_k-\matr{D}_{k-1}\|
\geq\kappa_{\mathrm{D}}
\left\|
\nabla f_{\matr{A}_{k-1}}\!\left(\matr{D}_{k-1}\right)
\right\|,
\]
where $\kappa_{\mathrm{D}}=1/K_{\mathrm{D}}$. This proves part~(iii).
\end{proof}

\section{Proof of the indirect-fit update lemma}
\label{app:proof-indirect-a}

\begin{proof}[Proof of \cref{lem:A-update-indirect}]
Recall that $h_k$ is the elementary function defined in
\eqref{eq:hk_indirect}. If $\operatorname{dist}(0,\partial h_k(0))=0$, then
$\alpha_k^*=0$ by \hyperlink{rule:R1}{Rule~(R1)}, and all three
statements hold trivially. Hence, in the remainder of the proof, we assume
$\operatorname{dist}(0,\partial h_k(0))>0.$ 

We first prove part~(i). Fix an update iteration $k$ and set
$p=p_k$, $i=i_k$, and $j=j_k$. For a trial step $\alpha$, let $\matr{A}(\alpha)$
be obtained from $\matr{A}_{k-1}$ by replacing its
$p$-th column with
$\matr{T}_{i,j}(\alpha)(\matr{A}_{k-1})_{:,p}$. The ECT transformation changes
only two entries:
\[
(\matr{A}(\alpha))_{i,p}
=(\matr{A}_{k-1})_{i,p}+\alpha(\matr{A}_{k-1})_{j,p},
\\ \ 
(\matr{A}(\alpha))_{j,p}
=(1-\alpha)(\matr{A}_{k-1})_{j,p}.
\]
Set $a\eqdef(\matr{A}_{k-1})_{i,p}\geq0$ and
$b\eqdef(\matr{A}_{k-1})_{j,p}>0$, so the feasible interval is
$[-a/b,1]$.
Define
$\matr{Q}^{(\ell)}_{k-1}\eqdef
\matr{A}_{k-1}^{\T}(\matr{C}^{(\ell)})^{-1}$ and
$\matr{P}^{(\ell)}_{k-1}\eqdef\matr{Q}^{(\ell)}_{k-1}\matr{A}_{k-1}$. A direct expansion under the ECT update gives the quadratic polynomial
\begin{equation}\label{eq:h_quadratic}
h_k(\alpha)-h_k(0)=c_2\alpha^2+c_1\alpha,
\end{equation}
with the coefficients given by
\begin{align*}
c_2=2b^2\sum_{\ell=1}^{L}\sum_{\substack{q=1\\q\neq p}}^{m}
\bigl(Q_{q,i}^{(\ell)}-Q_{q,j}^{(\ell)}\bigr)^2,\ \ 
c_1
=4b\sum_{\ell=1}^{L}\sum_{\substack{q=1\\q\neq p}}^{m}
P_{q,p}^{(\ell)}
\bigl(Q_{q,i}^{(\ell)}-Q_{q,j}^{(\ell)}\bigr),
\end{align*}
where $Q_{q,i}^{(\ell)}$ and $P_{q,p}^{(\ell)}$ denote the corresponding entries of $\matr{Q}^{(\ell)}_{k-1}$ and $\matr{P}^{(\ell)}_{k-1}$, respectively.

By \eqref{eq:h_quadratic} and \cref{rem:dist_equals_derivative}, we have 
$h'_k(0)=c_1$ and  
$\operatorname{dist}\!\left(0,\partial h_k(0)\right)=|h'_k(0)|=|c_1|.$ 
If $c_2=0$, the coefficient formulas give $c_1=0$, contradicting the
standing assumption. Thus, we have $c_2>0$.
The unconstrained minimizer of \eqref{eq:h_quadratic} is
$-c_1/(2c_2)$. The constrained minimizer is its projection onto
$[-a/b,1]$. Since this interval contains zero, the exact line search gives
$|\alpha_k^*|\leq |c_1|/(2c_2)$. Moreover, \hyperlink{rule:R1}{Rule~(R1)} follows from the descent
direction that $c_1\alpha_k^*=-|c_1||\alpha_k^*|$. Therefore, we obtain
\begin{align*}
h_k(0)-h_k(\alpha_k^*)
=-c_2(\alpha_k^*)^2-c_1\alpha_k^*
=|c_1||\alpha_k^*|-c_2(\alpha_k^*)^2
\geq\frac{1}{2}|c_1||\alpha_k^*|
=\eta_{\mathrm{A},2}\operatorname{dist}\!\left(
0,\partial h_k(0)\right)|\alpha_k^*|.
\end{align*}
This proves part~(i).

Part~(ii) follows directly from \hyperlink{rule:R1}{Rule~(R1)}.

Finally, we prove part~(iii).
Suppose that \cref{assump:interiority} holds for the $\matr{A}$-block.
For all sufficiently large $k$, $\alpha_k^*$ cannot attain either endpoint
of the feasible interval, since $\alpha_k^*=1$ would give
$(\matr{A}_k)_{j,p}=0$, while $\alpha_k^*=-a/b$ would give
$(\matr{A}_k)_{i,p}=0$.
Hence, the exact line search yields the interior minimizer
$\alpha_k^*=-c_1/(2c_2)$. 
Define
$\tilde{c}_2\eqdef c_2/b^2$.
Since $\matr{A}_{k-1}\in\Delta_n^m$, the matrices
$(\matr{C}^{(\ell)})^{-1}$ are fixed, and there are finitely many index
pairs, $2\tilde{c}_2$ is bounded above by a constant
$K_{\mathrm{A},2}>0$ independent of $k$.
With $\kappa_{\mathrm{A},2}=1/K_{\mathrm{A},2}$, we obtain
\[
|\alpha_k^*|
=\frac{|c_1|}{2b^2\tilde{c}_2}
\geq\frac{|c_1|}{b^2K_{\mathrm{A},2}}
\geq\kappa_{\mathrm{A},2}\operatorname{dist}\!\left(0,\partial h_k(0)\right),
\]
where the last inequality uses $0<b\leq1$. This proves part~(iii).
\end{proof}

\section{Proofs of the Semi-NNICA convergence corollaries}
\label{app:proof-semi-nnica-convergence}

We prove \cref{corollary:direct_fit,corollary:indirect_fit} by verifying
the hypotheses of \cref{thm:Schneider_frame_two} using the model-specific
update estimates.

\begin{proof}[Proof of \cref{corollary:direct_fit}]
Apply \cref{thm:Schneider_frame_two} with $f=f_{\mathrm{dir}}$.
For fixed $\lambda>0$, this objective is a polynomial in the entries of
$\matr{A}$ and the diagonal factors, so it admits a real-analytic extension to the ambient
Euclidean space. Moreover, $\Omega=(\set{D}_m)^L$ is a closed linear subspace. The map
\[
\tau_{\matr{V}}:(\mathbb{R}^{m})^{L}\to\Omega,
\\ \ 
(\vect{d}^{(1)},\ldots,\vect{d}^{(L)})
\mapsto\left(\operatorname{Diag}(\vect{d}^{(\ell)})\right)_{\ell=1}^{L},
\]
is a real-analytic linear homeomorphism onto $\Omega$ and satisfies the
conditions of \cref{lem:nonsmooth-lojasiewicz}. 
By \cref{prop:boundedness-direct}, the full sequence remains in a compact
sublevel set of $f_{\mathrm{dir}}$ and therefore admits an accumulation point.
Parts~(i)--(ii) of \cref{lem:A-update-direct,lem:D-update-direct} verify
\eqref{cond:U1}--\eqref{cond:U2} and
\eqref{cond:V1}--\eqref{cond:V2}, respectively, with
$\eta_{\mathrm{U}}=\eta_{\mathrm{A},1}$ and
$\eta_{\mathrm{V}}=\eta_{\mathrm{D}}$.
The block-selection rule uses the gradient of $f_{\mathrm{dir}}$, so these
estimates apply to the same objective as in the general theorem.
Thus, the first conclusion
of \cref{thm:Schneider_frame_two} gives convergence of the whole sequence.
Since $\Omega$ is a linear subspace, we have
$g^-_{f_{\matr{A}}}(\matr{D})
=\|\nabla f_{\matr{A}}\!\left(\matr{D}\right)\|$, 
which is continuous. If \cref{assump:interiority} also holds, part~(iii) of the two lemmas verifies
\eqref{cond:U3} and \eqref{cond:V3} with
$\kappa_{\mathrm{U}}=\kappa_{\mathrm{A},1}$ and
$\kappa_{\mathrm{V}}=\kappa_{\mathrm{D}}$.
The final conclusion of the theorem then
gives constrained stationarity for $f_{\mathrm{dir}}$ and the stated rate with
its constrained \L{}ojasiewicz exponent.
\end{proof}

\begin{proof}[Proof of \cref{corollary:indirect_fit}]
Because the matrices $(\matr{C}^{(\ell)})^{-1}$ are fixed, the indirect-fit
objective is a polynomial in the entries of $\matr{A}$. Represent the absent
second block by the closed singleton
$\Omega=\{\matr{0}_{1\times1}\}\subseteq\mathbb{R}^{1\times1}$ and define the real-analytic
parametrization $\tau_{\matr{V}}:\mathbb{R}\to\Omega$ by
$\tau_{\matr{V}}(z)=\matr{0}_{1\times1}$. Its image is $\Omega$ itself, so the
analytic-parametrization conditions in \cref{lem:nonsmooth-lojasiewicz} hold. 
Compactness of $\Delta_n^m$ gives an accumulation point. The dummy
$\matr{V}$-block is constant, so \eqref{cond:V1}--\eqref{cond:V3} hold
trivially. Parts~(i)--(ii) of \cref{lem:A-update-indirect} verify
\eqref{cond:U1}--\eqref{cond:U2} with
$\eta_{\mathrm{U}}=\eta_{\mathrm{A},2}$, and the first conclusion of
\cref{thm:Schneider_frame_two} gives global convergence. If
\cref{assump:interiority} also holds, part~(iii) verifies \eqref{cond:U3} with
$\kappa_{\mathrm{U}}=\kappa_{\mathrm{A},2}$;
the final conclusion of the theorem then gives stationarity and the stated
rate.
\end{proof}

\noindent\textbf{Acknowledgments.}
This work was supported by the Guangdong Basic and Applied Basic
Research Foundation (2026A1515050001, 2026A1515010188).

\bibliographystyle{siamplain}
\bibliography{Tensor_All}

\end{document}